\documentclass[12pt]{article}

\usepackage{amssymb,bm}
\usepackage{amsmath}
\usepackage{amsthm}
\usepackage{graphicx}
\usepackage[active]{srcltx} 
\usepackage{hyperref} 
\hypersetup{pdfborder=0 0 0}
\newtheorem{theorem}{{\sc Theorem}}[section]

\newtheorem{lemma}[theorem]{{\sc Lemma}}
\newtheorem{corollary}[theorem]{Corollary}
\newtheorem{remark}[theorem]{Remark}

\newtheorem{definition}[theorem]{Definition}

\newcommand{\bb}[1]{\mathbb{ #1}}

\bmdefine\Bone{1}

\newcommand{\dev}[1]{\mathrm{dev}(#1)}

\newcommand{\dOm}{\partial\Omega}

\newcommand{\bra}[1]{\overline{#1}}

\newcommand{\Trc}{\mathrm{Tr}\,}

\newcommand{\rank}{\mathrm{rank}}
\newcommand{\tns}[1]{#1\otimes #1}
\newcommand{\hf}{\displaystyle\frac{1}{2}}
\newcommand{\nth}[1]{\displaystyle\frac{1}{#1}}

\newcommand{\dif}[2]{\displaystyle\frac{\partial #1}{\partial #2}}
\newcommand{\Grad}{\nabla}
\newcommand{\Div}{\nabla \cdot}

\newcommand{\Md}{\partial}

\newcommand{\Tld}[1]{\widetilde{#1}}

\newcommand{\av}[1]{\langle #1 \rangle}

\def\Xint#1{\mathchoice
{\XXint\displaystyle\textstyle{#1}}%
{\XXint\textstyle\scriptstyle{#1}}%
{\XXint\scriptstyle\scriptscriptstyle{#1}}%
{\XXint\scriptscriptstyle\scriptscriptstyle{#1}}%
\!\int}
\def\XXint#1#2#3{{\setbox0=\hbox{$#1{#2#3}{\int}$ }
\vcenter{\hbox{$#2#3$ }}\kern-.6\wd0}}

\def\dashint{\Xint-}

\newcommand{\re}{\Re\mathfrak{e}}
\newcommand{\jump}[1]{\lbrack\!\lbrack #1 \rbrack\!\rbrack}
\newcommand{\lump}[1]{\lbrace\skew{-14.7}\lbrace\!\!#1\!\!\skew{14.7}\rbrace\rbrace}

\newcommand{\bc}{boundary condition}

\newcommand{\rhs}{right-hand side}
\newcommand{\lhs}{left-hand side}
\newcommand{\mc}{microstructure}

\newcommand{\nbh}{neighborhood}
\newcommand{\IFF}{if and only if }

\newcommand{\Ga}{\alpha}
\newcommand{\Gb}{\beta}
\newcommand{\Gd}{\delta}
\newcommand{\Ge}{\epsilon}

\newcommand{\Gve}{\varepsilon}

\newcommand{\Gk}{\kappa}

\newcommand{\Gth}{\theta}

\newcommand{\Go}{\omega}

\newcommand{\GD}{\Delta}

\newcommand{\GG}{\Gamma}

\newcommand{\GO}{\Omega}

\bmdefine\BGa{\alpha}
\bmdefine\BGb{\beta}
\bmdefine\BGd{\delta}
\bmdefine\BGe{\epsilon}
\bmdefine\BGve{\varepsilon}
\bmdefine\BGf{\phi}
\bmdefine\BGvf{\varphi}
\bmdefine\BGg{\gamma}
\bmdefine\BGc{\chi}
\bmdefine\BGi{\iota}
\bmdefine\BGk{\kappa}
\bmdefine\BGl{\lambda}
\bmdefine\BGn{\eta}
\bmdefine\BGm{\mu}
\bmdefine\BGv{\nu}
\bmdefine\BGp{\pi}
\bmdefine\BGth{\theta}
\bmdefine\BGvth{\vartheta}
\bmdefine\BGr{\rho}
\bmdefine\BGvr{\varrho}
\bmdefine\BGs{\sigma}
\bmdefine\BGvs{\varsigma}
\bmdefine\BGt{\tau}
\bmdefine\BGj{\tau}
\bmdefine\BGu{\upsilon}
\bmdefine\BGo{\omega}
\bmdefine\BGx{\xi}
\bmdefine\BGy{\psi}
\bmdefine\BGz{\zeta}
\bmdefine\BGD{\Delta}
\bmdefine\BGF{\Phi}
\bmdefine\BGG{\Gamma}
\bmdefine\BGL{\Lambda}
\bmdefine\BGP{\Pi}
\bmdefine\BGT{\Theta}
\bmdefine\BGS{\Sigma}
\bmdefine\BGU{\Upsilon}
\bmdefine\BGO{\Omega}
\bmdefine\BGX{\Xi}
\bmdefine\BGY{\Psi}

\bmdefine\BFM{\mathfrak{M}}
\bmdefine\BFb{\mathfrak{b}}
\bmdefine\BFk{\mathfrak{k}}
\bmdefine\BFm{\mathfrak{m}}
\bmdefine\BFu{\mathfrak{u}}
\bmdefine\BFv{\mathfrak{v}}

\newcommand{\CA}{{\mathcal A}}

\newcommand{\CE}{{\mathcal E}}

\newcommand{\CO}{{\mathcal O}}

\bmdefine\BCA{{\mathcal A}}
\bmdefine\BCB{{\mathcal B}}
\bmdefine\BCC{{\mathcal C}}
\bmdefine\BCD{{\mathcal D}}
\bmdefine\BCE{{\mathcal E}}
\bmdefine\BCF{{\mathcal F}}
\bmdefine\BCG{{\mathcal G}}
\bmdefine\BCH{{\mathcal H}}
\bmdefine\BCI{{\mathcal I}}
\bmdefine\BCJ{{\mathcal J}}
\bmdefine\BCK{{\mathcal K}}
\bmdefine\BCL{{\mathcal L}}
\bmdefine\BCM{{\mathcal M}}
\bmdefine\BCN{{\mathcal N}}
\bmdefine\BCO{{\mathcal O}}
\bmdefine\BCP{{\mathcal P}}
\bmdefine\BCQ{{\mathcal Q}}
\bmdefine\BCR{{\mathcal R}}
\bmdefine\BCS{{\mathcal S}}
\bmdefine\BCT{{\mathcal T}}
\bmdefine\BCU{{\mathcal U}}
\bmdefine\BCV{{\mathcal V}}
\bmdefine\BCW{{\mathcal W}}
\bmdefine\BCX{{\mathcal X}}
\bmdefine\BCY{{\mathcal Y}}
\bmdefine\BCZ{{\mathcal Z}}

\bmdefine\Bzr{ 0}
\bmdefine\Ba{ a}
\bmdefine\Bb{ b}
\bmdefine\Bc{ c}
\bmdefine\Bd{ d}
\bmdefine\Be{ e}
\bmdefine\Bf{ f}
\bmdefine\Bg{ g}
\bmdefine\Bh{ h}
\bmdefine\Bi{ i}
\bmdefine\Bj{ j}
\bmdefine\Bk{ k}
\bmdefine\Bl{ l}
\bmdefine\Bm{ m}
\bmdefine\Bn{ n}
\bmdefine\Bo{ o}
\bmdefine\Bp{ p}
\bmdefine\Bq{ q}
\bmdefine\Br{ r}
\bmdefine\Bs{ s}
\bmdefine\Bt{ t}
\bmdefine\Bu{ u}
\bmdefine\Bv{ v}
\bmdefine\Bw{ w}
\bmdefine\Bx{ x}
\bmdefine\By{ y}
\bmdefine\Bz{ z}
\bmdefine\BA{ A}
\bmdefine\BB{ B}
\bmdefine\BC{ C}
\bmdefine\BD{ D}
\bmdefine\BE{ E}
\bmdefine\BF{ F}
\bmdefine\BG{ G}
\bmdefine\BH{ H}
\bmdefine\BI{ I}
\bmdefine\BJ{ J}
\bmdefine\BK{ K}
\bmdefine\BL{ L}
\bmdefine\BM{ M}
\bmdefine\BN{ N}
\bmdefine\BO{ O}
\bmdefine\BP{ P}
\bmdefine\BQ{ Q}
\bmdefine\BR{ R}
\bmdefine\BS{ S}
\bmdefine\BT{ T}
\bmdefine\BU{ U}
\bmdefine\BV{ V}
\bmdefine\BW{ W}
\bmdefine\BX{ X}
\bmdefine\BY{ Y}
\bmdefine\BZ{ Z}

\newcommand{\SFC}{\mathsf{C}}

\usepackage{color}

\title{A class of nonlinear elasticity problems with no local  but many global minimizers
}
\author{\em Yury Grabovsky \and\em Lev Truskinovsky}

\begin{document}
\maketitle

\hspace{54ex}{\footnotesize\emph{Dedicated to 85th birthday of R. Fosdick}}

\begin{abstract}
  We present  a  class of  models of  elastic phase
  transitions with incompatible energy wells in any space dimension,
  where  an   abundance of Lipschitz global minimizers in a hard device
  coexists with a complete lack of strong local minimizers.
The analysis hinges on the   proof that every
  strong local minimizer in a hard device is also a global
  minimizer which is applicable   much  beyond the chosen   class of models.  Along the way we show that  a  new proof of sufficiency for  a subclass of  affine boundary conditions  can be  built   around a  novel  nonlinear generalization of the   classical Clapeyron theorem,   whose subtle relation to dynamics was studied extensively by  R. Fosdick.

\end{abstract}

\section{Introduction}
\setcounter{equation}{0}
\label{sec:intro}
The phenomenon of \emph{metastability} in elastostatics, manifesting itself through
the existence of strong local minimizers that are not global\footnote{   The implied interpretation of the physical concept of \emph{metastability}  is not the only one possible. In physics literature the term \emph{metastability} usually refers to the presence
    of configurations that are stable only under sufficiently small
    perturbations. Here we specify the notion of ``small'' and interpret 
     strong
    local minimizers   as configurations minimizing among competitors  with deformation fields that 
    are uniformly close. A more conventional weak local minimizers  are configurations minimizing among
    the competitors with not only    deformations but also  deformation gradients  that are  uniformly
    close. 
     }, is usually associated  with Neumann
boundary conditions  (soft device) and linked to the incompatibility of the energy wells
\cite{baja15}. In Dirichlet  problem (hard device) the situation is different  and the incompatibility  of the energy wells is not sufficient  for \emph{metastability}.  To corroborate this statement   we present here  a class of hyperelastic materials with incompatible energy wells 
  for which one can prove the absence of \emph{metastability}  for any Dirichlet \bc s (hard device).  
  
   Moreover,  we prove for  the associated  class of vectorial variational problems  the absence of strong local minimizers which are not global on any domain and
in any number of dimensions.  For a subclass of such problems, amenable to an especially  transparent analytically
  treatment,   we exhibit a concurrent  
dramatic nonuniqueness of global minimizers.  
While the    simplest  elasticity problem, exhibiting  in a hard device both the    abundance of global minimizers  and  the  lack of strong local  minimizers, is a scalar problem  in one dimension \cite{erick75}, the question whether the effect survives in a multidimensional vectorial setting has been so far open. The general fact that  the elastic energy relaxation can in principle produce 
 a  wild non-uniqueness of optimal microstructures 
 is  well-known \cite{DW,gra}, however, whether such non-uniqueness   extends or not to strong local minimizers is usually difficult to check.

We recall that  to ensure uniqueness in nonlinear elasticity, the use of
Dirichlet boundary conditions and the presence of topological, or even geometric simplicity of the domain,  are essential
\cite{posi97,tahe05,spad09}. 
Uniqueness has been
established 
for  star-shaped domains, affine
displacement boundary conditions, and strictly quasiconvex stored energy
functions \cite{knst84,tahe03}. It was also understood that whether uniqueness holds may  depend on the
regularity class in which one looks for a minimizer \cite{ktw03}, as well as  its  integrability class \cite{ball82,sivo92,hopo95}. 
Instead, for mixed, Dirichlet-Neumann  boundary-value problems of nonlinear elasticity,
 nonuniqueness   has been found ubiquitous  with  the most familiar examples being those
 associated with buckling and related  to the  emergence   of multiple symmetry-related energy minima  
\cite{eul1744}, but not only \cite{grtr07}.
The possibility of   nonuniqueness  with
  purely Dirichlet \bc s is usually    associated with   quenched  inhomogeneity
as in  problems with residual stresses  or non-zero body tractions 
 \cite{edfo68,ledr87,trma85}. 

 The interest of a  simple example considered in this paper stems from the fact  that  it exhibits multiplicity
 of  global minimizers  in a hard device problem   in the
 absence of any geometrical  complexity of the domain and does not require
  quenched inhomogeneity. Furthermore, the obtained non-uniqueness is not due
 to the scalar nature of the problem (absence of compatibility constraint)  and  is unrelated to symmetry related degeneracy as in the case of Euler buckling. 
   Instead,  it is evocative of  nonuniqueness in the problems of  \mc\ optimization  in composites
 \cite{a,MK,gra,LU}.

In this paper we limit our attention to the class of  ``geometrically
linearized''  Hadamard materials which can be viewed as   simplified analogs   of the fully nonlinear Hadamard materials \cite{hada03,john66}.  The  energy density functions of the latter are of the form
\begin{equation}
  \label{Hadamard}
  W(\BF)=\mu|\BF|^{2} +h(\det\BF).
\end{equation}
Here $\BF$ is the deformation gradient and $\mu$ is the measure of rigidity. The non-negative function $h(d)$ is defined on $(0,+\infty)$ and has the
property that $h(d)\to\infty$, as $d\to 0^{+}$.  If we now formally use for the materials   \eqref{Hadamard}
the   ``geometric'' approximation  $\det\BF\approx 1+\Trc(\BF-\BI)$,   valid in the limit $\BF\to\BI$,  a   formal asymptotic expansion with respect
to a small parameter would necessarily also imply the  ``physical'' linearization  which    trivializes the problem. Since our goal 
is to  retain and exploit physical non-linearity, we consider instead the energy
\begin{equation}
  \label{Hadamod}
  W(\BH)=g(\Trc\BH)+\mu|\BH|^{2},\qquad \BH=\BF-\BI.
\end{equation}
which is formally unrelated to \eqref{Hadamard}. However, following  \cite{kh,buhala83,karo88,abji89I,abji89II}, we can  perform geometric linearization in the second term in  \eqref{Hadamod}, replacing
$\BH$ with its symmetric part $\BGve=(\BH+\BH^{t})/2$ and, to emphasize the  isotropic nature of the resulting energy density,     write the result  in the form 
\begin{equation}
  \label{ex:ener}
  W_{0}(\BH)=f(\Trc\BGve)+\mu|\dev{\BGve}|^{2},\qquad \dev{\BGve}=\BGve-\nth{d}\Trc(\BGve)\BI.
\end{equation}
Despite the achieved formal simplification, the mathematical features  of interest are similar in the equilibrium  problems  for materials with the energies (\ref{Hadamod}) and (\ref{ex:ener}). We will therefore, focus on (\ref{ex:ener}), knowing that  it also preserves  at least the linearized version of the frame indifference  property.

\begin{figure}
\centering
    \includegraphics[scale=0.3]{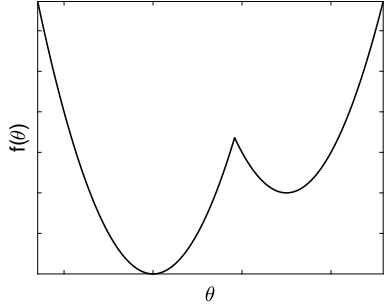} 
\label{fig:ftheta}
\caption{Double-well nonlinearity in a geometrically linear bi-quadratic Hadamard material.}
\end{figure}
  
We assume that $\mu>0$ and choose the function $f$ to be nonconvex and have  a 
``double-well'' shape so that  the full
energy has incompatible  energy wells.
For additional analytical transparency  we  will use in our explicit constructions the bi-quadratic potential 
\begin{equation}
 \label{linisel}
 f(\Gth)=\min\{\Gk_{0}\Gth^{2},\Gk_{0}(\Gth-\Gth_{p})^{2}+f_{0}\}, 
\end{equation}
illustrated in Fig.~\ref{fig:ftheta}. The ensuing model describes a material capable of undergoing a purely dilatational phase transformation
between elastic phases which are both linearly elastic and have  the same  elastic  moduli but  different reference (chemical) energies.   The relaxation in such elasticity problems is nontrivial and historically  has been viewed as a benchmark case for various mathematical theories of elastic phase transitions \cite{kh,DW,pipk91}.

If the energy wells are rank-one connected, the   quasiconvexification
of a non-convex  energy can be expected to coincide with rank-one convexification and in many cases even with convexification.
Instead, the incompatibility of the wells
can be responsible for a nontrivial structure of the quasiconvex envelope. 
  In this respect the  material (\ref{Hadamod})  is of interest  
 because, even though  the non-convexity enters
\eqref{ex:ener} only through a scalar potential, the energy wells remain not
rank-one connected. 
This means, for instance, that  even  if the energy wells in \eqref{Hadamod}
  are both at zero energy,  not only the relaxed energy  is nonzero but also the  optimal  microstructures are highly nontrivial.

The energy density (\ref{ex:ener}) is  non
 rank-one convex for all $\mu>0$, as long as  $f(\Gth)$ is given by (\ref{linisel})  \cite{DW}. Materials  with such energies,  loaded in a soft device,   are expected to exhibit   constitutive  hysteresis 
 \cite{bchj95,jwk19,rosl97,silh03,zjm09}. Instead,  in the hard device,   the strong local
 minima in the model \eqref{ex:ener} are absent, the quasistatic mechanical response 
 is defined uniquely  and the deformation path is reversible as it  corresponds to global minimization of the total energy at each value of the loading parameter.
The associated stress-strain relation can be obtained from  the knowledge of
the  quasiconvex envelope of the  energy \eqref{ex:ener};  the latter  can be
characterized explicitly in the whole range of parameters   for all
double-well shaped potentials\footnote{ In fact, the
    quasiconvex envelope of the  energy can be characterized for all
    continuous potentials
  $f(\Gth)$ that are bounded from below.} $f(\Gth)$. Here, to make the paper self-contained,   we compute  the   relaxation of the energy  \eqref{ex:ener} using simple laminates as the  optimal microstructure; a different proof of the same result,  based on   matching of upper and lower bounds,  was previously given  in  \cite{grtrpcx}.

The most important  general (not energy-specific) result obtained  in  this paper is the formulation of the far-reaching necessary
and sufficient conditions for global minimizers in star-shaped domains with
affine boundary conditions. This result provide new   tools for  characterizing  energy minimizing configurations in examples where the energy relaxation, or at least
the elastic binodal, is explicitly known. An important technical lemma in the associated 
analysis is a fundamental nonlinear generalization of the classical Clapeyron theorem\footnote{The   necessity of Clapeyron theorem-type formula in the study of uniqueness in the Dirichlet setting was first realized   by  Knops and Stuart in
  \cite{knst84}.}
which links the energy of the equilibrium configuration with the work of the (generalized) forces on the boundary of the body. 
Combining Clapeyron's theorem with affine \bc s and local
    material stability (quasiconvexity) we obtain the inequality establishing
    global minimality of all stationary locally stable configurations in
    star-shaped domains. All these results may be viewed, in a more general perspective, as an extension
  of optimality conditions for extremal \mc s in composites (e.g. \cite{gra}). 

As an application of the   obtained necessary and sufficient conditions we show that  one can reduce the problem of finding a global minimizer for the energy (\ref{ex:ener}) to the solution of a nonlinear free boundary problem involving a system of  linear PDEs 
 in a finite domain. It is remarkable  that  the solution of the ensuing problem  can be found explicitly 
 for any double-well scalar potential $f(\Gth)$. The degeneracy in this problem, due to the fact  the  optimality conditions  
 do not place very stringent constraints on the displacement gradient in an
 optimal configuration,  is behind the 
 multiplicity of global minimizers. We show however, in Appendix~\ref{sec:MOMA},  that  topologically
 simple energy minimizers which maintain the symmetry of the domain may fail to exist in some domains with corners.

Our  analysis of optimal microstructures suggests that, despite the unavailability of the constitutive   hysteresis in the considered  class of problems,   which lack  local minimizers, the direct and the reverse
transformation may follow morphologically different transformation paths in the real physical 
space. Therefore, despite  the absence of
\emph{metastability},
 one may  encounter a   phenomenon of  ``morphological hysteresis'', whereby the loading and unloading may occur along different morphological paths.  The fact, that along such paths the system
traverses distinct but energetically equivalent configurations,
may  be of interest in applications if
  one could design a way  of biasing  one of the paths  through, say,  imposing
  additional control on the gradients.

The paper is organized as follows. In Section \ref{sec:slm} we
 compute the quasiconvexification of the energy density
  (\ref{ex:ener}), and use the result to prove the absence of strong local
  minimizers that are not global in any domain and for any Dirichlet type \bc s.
The necessary and sufficient conditions for global minimizers are formulated in
Section \ref{sec:necsuf}. The issue of attainment and the multiplicity of
global minimizers in our model  are discussed in Section \ref{sec:attainment}.
The last  Section \ref{sec:conc} summarizes  the  results and
presents conclusions. Our Appendix \ref{sec:ac} contains a technical
discussion of the  generic  loss of strict ellipticity  of the
rank one convexification of an energy density. In Appendix~\ref{sec:MOMA} we
show the non-existence of square-symmetric simply connected inclusion
minimizing the energy in the square.

\section{Local minimizers}
 \setcounter{equation}{0}
\label{sec:slm}
Given that the  case of interest  involves hard device loading, our goal is
to study strong local and global minima of the functional
\begin{equation}
\label{hdener}
\CE_{0}(\Bu)=\int_{\GO}
\left\{f(\Div\Bu)+
\mu\left|e(\Bu)-\nth{d}(\Div\Bu)\BI\right|^2\right\}d\Bx
\end{equation}
among all Lipschitz displacement vector fields $\Bu\in
W^{1,\infty}(\GO;\bb{R}^{d})$ subject to the  constraint
\begin{equation}
\label{Dbc}
\Bu(\Bx)=\Bu_{0}(\Bx),\quad \Bx\in\dOm,
\end{equation}
where $\Bu_{0}(\Bx)$ is a given Lipschitz function on $\dOm$. Here
$\GO\subset\bb{R}^{d}$ denotes a Lipschitz domain. Restricting attention to
Lipschitz minimizers allows us to focus on instabilities caused by failure of
quasiconvexity \cite{morr52,daco82} and exclude some other instabilities
related, for instance, to the mismatch between the integrability of the
minimizing sequences and the growth of the energy density at infinity
\cite{bamu84}.

First we recall that one of the necessary conditions of strong local
minimality is stability with respect to nucleation of a coherent precipitate
of the new phase in the interior of the old phase. That means that we must
have $\Grad\Bu(\Bx)\in\BCA$, for a.e. $\Bx\in\GO$,
\cite[Proposition~4.1]{tahe05},
where
\begin{equation}
  \label{admset}
  \BCA=\{\BH\in\bb{M}:W_{0}(\BH)=QW_{0}(\BH)\}, 
\end{equation}
and where $\bb{M}$ denotes the set of all $d\times d$ matrices and $QW_{0}(\BH)$ is
the quasiconvexification of $W_{0}$ \cite{daco82}.

In this paper we show that if $\Bu(\Bx)$ is a weak solution of the
Euler-Lagrange equation for the energy (\ref{hdener}) with the \bc s
(\ref{Dbc}) and satisfies $\Grad\Bu(\Bx)\in\BCA$ for a.e. $\Bx\in\GO$ then it
is necessarily a global minimizer for $\CE_{0}(\Bu)$. However, before we move
to the actual analysis, it is appropriate to mention that we are not aware of
any examples of proper metastable states in a hard device in domains with
trivial topology. Moreover it has been proved for general energies that every
strong local minimizer is global in the hard device with affine Dirichlet \bc
s in star-shaped domains \cite{tahe03}. In our special example, the same
result turns out to be true for all Dirichlet \bc s in domains with piecewise smooth
boundaries and arbitrary topology.
 
The first step of the analysis is to characterize the set $\BCA$ of admissible displacement gradients for the  chosen material  model. It can result from
  the computation of $QW_{0}(\BH)$   which  in our case  can be done explicitly.  
\begin{theorem}
  \label{th:exqcx}
Suppose that $W_{0}(\BH)$ is given by (\ref{ex:ener}). Then,
  \begin{equation}
    \label{ex:qcx}
    QW_{0}(\BH)=\Phi^{**}(\Trc\BH)+\frac{\mu}{4}|\BH-\BH^{t}|^{2}-2\mu J_2(\BH),
  \end{equation}
where 
\begin{equation}
\Phi(\Gth)=f(\Gth)+\frac{d-1}{d}\mu\Gth^2,
\label{Fdef}
\end{equation}
\begin{equation}
2J_2(\BH)=(\Trc\BH)^{2}-\Trc(\BH^{2})=2\sum_{i<j}\left|
\begin{array}{rr}
h_{ii} & h_{ij}\\
h_{ji} & h_{jj}
\end{array}
\right|
\label{J2inv}
\end{equation}
is a quadratic null-Lagrangian, and $\Phi^{**}$ denotes the convexification of
$\Phi(\Gth)$.
\end{theorem}
\begin{proof}
It will be convenient to start with the well known definition of quasiconvexification
of an arbitrary energy density function $W(\BH)$
using periodic functions \cite{daco82}:
\begin{equation}
  \label{qcxper}
  QW(\BH)=\inf_{\BGf-Q\text{-periodic}}\dashint_{Q}W(\BH+\Grad\BGf(\Bx))d\Bx,
\end{equation}
where $Q=[0,1]^{d}$.
The computation of $QW_{0}$ is based on
the following formula
\begin{equation}
\left|e(\Bu)-\nth{d}(\Div\Bu) \BI\right|^2=\frac{d-1}{d}(\Div\Bu)^2+
\nth{4}|\Grad\Bu-(\Grad\Bu)^{t}|^2-2J_2(\nabla\Bu).
\label{translation}
\end{equation}
 Let $\Bu(\Bx)=\BH\Bx+\BGf(\Bx)$. Then
\begin{equation}
\dashint_{Q}W_{0}(\Grad\Bu)d\Bx=
-2\mu J_2(\BH)+\dashint_{Q}\left(\Phi(\Div\Bu)+\frac{\mu}{4}|\Grad\Bu-(\Grad\Bu)^{t}|^2\right)d\Bx.
\label{enerform}
\end{equation}
Since $\Phi^{**}(\Gth)\le\Phi(\Gth)$ for all $\Gth\in\bb{R}$, we obtain the inequality
\[
\dashint_{Q}W_{0}(\Grad\Bu)d\Bx\ge -2\mu J_2(\BH)+\dashint_{Q}\left(\Phi^{**}(\Div\Bu)+
\frac{\mu}{2}|\Grad\Bu-(\Grad\Bu)^{t}|^2\right)d\Bx.
\]
Now applying the Jensen's inequality we get:
\begin{equation}
\dashint_{Q}W_{0}(\Grad\Bu)d\Bx\ge-2\mu J_2(\BH)+\Phi^{**}(\Trc\BH)+\frac{\mu}{4}|\BH-\BH^{t}|^{2}.
\label{ex:keyineq}
\end{equation}
Therefore,
\begin{equation}
  \label{QWlb}
QW_{0}(\BH)\ge -2\mu J_2(\BH)+\Phi^{**}(\Trc\BH)+\frac{\mu}{4}|\BH-\BH^{t}|^{2}.
\end{equation}
In order to prove equality, we need to exhibit a periodic function $\BGf$
such that $\Grad\BGf=(\Grad\BGf)^{t}$ and
\begin{equation}
  \label{bestphi}
  \dashint_{Q}\Phi(\Div\BGf(\Bz)+\Trc\BH)d\Bz=\Phi^{**}(\Trc\BH).
\end{equation}
\begin{figure}
\begin{center}
\includegraphics[scale=0.4]{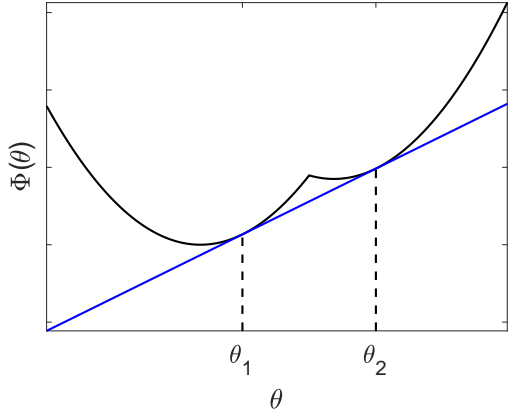}
\end{center}
\caption{Definition of $\Gth_{1}<\Gth_{2}$ via the common tangent construction.}
\label{fig:Phitheta}
\end{figure}
If $\Phi(\Gth)$, given by (\ref{Fdef}) is convex\footnote{This is never the case for the potential (\ref{linisel}).}, then the original energy $W_{0}(\BH)$ is quasiconvex (polyconvex, actually), and formula (\ref{ex:qcx}) holds. The case of interest is therefore when $\Phi(\Gth)$ retains the double-well shape,
in which case $\Phi^{**}(\Gth)=\Phi(\Gth)$ \IFF
$\Gth\not\in(\Gth_{1},\Gth_{2})$, as in Fig.~\ref{fig:Phitheta}. For example, for the potential (\ref{linisel}) we find
\begin{equation}
  \label{th12}
  \Gth_{1}=\hf\left(\frac{f_{0}}{\Gk_{0}\Gth_{p}}+\frac{(d-1)\mu\Gth_{p}}{d\Gk_{0}+(d-1)\mu}\right),\qquad
  \Gth_{2}=\Gth_{p}+\hf\left(\frac{f_{0}}{\Gk_{0}\Gth_{p}}-\frac{(d-1)\mu\Gth_{p}}{d\Gk_{0}+(d-1)\mu}\right).
\end{equation}
Thus, if $\Trc\BH\not\in(\Gth_{1},\Gth_{2})$, then $\BGf=\Bzr$ attains equality in
(\ref{ex:keyineq}). 
If $\Trc\BH\in(\Gth_{1},\Gth_{2})$, then there exists $\Go\in(0,1)$ such that
\begin{equation}
\Trc\BH=\Go\Gth_1+(1-\Go)\Gth_2.
\label{lever_rule}
\end{equation}
In order to attain equality in (\ref{ex:keyineq}) we need to split the period
cell $Q=[0,1]^{d}$ into two subsets $A_{1}$, of volume $\Go$ and
$A_{2}$ of volume $1-\Go$ such that
\begin{equation}
  \label{peratt}
\Div\Bu(\Bx)=\Gth_{1}\chi_{A_{1}}(\Bx)+\Gth_{2}\chi_{A_{2}}(\Bx)
\end{equation}
for all $\Bx\in Q$. This is easily achieved with a laminate construction with
$A_{1}=[0,\Go]\times[0,1]^{2}$ and $A_{2}=[1-\Go,1]\times[0,1]^{2}$ and
function $\Grad\BGf$ being symmetric and piecewise constant:
\[
\Grad\BGf(\Bx)=\BM_{1}\chi_{A_{1}}(\Bx)+\BM_{2}\chi_{A_{2}}(\Bx).
\]
The symmetric matrices $\BM_{1}$ and $\BM_{2}$ are easily found from the periodicity
of $\BGf(\Bx)$, that is equivalent to $\dashint_{Q}\Grad\BGf(\Bx)dx=\Bzr$, and (\ref{peratt}):
\[
\BM_{1}=(1-\Go)\jump{\Gth}\tns{\Be_{1}},\qquad\BM_{2}=-\Go\jump{\Gth}\tns{\Be_{1}},
\]
where $\jump{\Gth}=\Gth_{1}-\Gth_{2}$. These choices guarantee equality in
(\ref{bestphi}).
\end{proof}

\smallskip

\begin{corollary}
  \begin{equation}
\BCA=\{\BH\in\bb{M}: \Trc\BH\not\in(\Gth_{1},\Gth_{2})\}.
\label{ex:binodal}
\end{equation}
\end{corollary}

  The constructed energy density $QW_{0}(\BH)$ is obviously polyconvex, and coincides with $RW_{0}(\BH)$, the
  rank-1 convexification of $W_{0}(\BH)$. Indeed, as we see from the proof of
  Theorem~\ref{th:exqcx}, the infimum in (\ref{qcxper}) can be achieved by a
  simple laminate. In view of the simplicity of the optimal microstructures,
  one can \emph{post factum} conclude that computing $QW_{0}(\BH)$ explicitly was
  not really necessary for the characterization of the set $\BCA$.  We could
  also delineate it by using inequality (\ref{ex:keyineq}) and relying on
  the knowledge and properties of the jump set introduced in
  \cite{grtrpe,grtrnc}. Indeed, inequality (\ref{ex:keyineq}) implies that all
  $\BH$ satisfying $\Phi^{**}(\Trc\BH)=\Phi(\Trc\BH)$ are admissible,
  i.e. $\{\BH\in\bb{M}: \Trc\BH\not\in(\Gth_{1},\Gth_{2})\}\subset\BCA$. We
  can then show that all $\BH$ with $\Gth_{1}<\Trc\BH<\Gth_{2}$ fail rank-one
  convexity by looking for pairs $\BH_{\pm}$ of the displacement gradient
  values that can occur on the two sides of a smooth jump discontinuity
  representing a phase boundary. We recall, \cite{grtrpe}, that these matrices
  satisfy the equations
\begin{equation}
  \label{js}
  \jump{\BH}=\Ba\otimes\Bn,\quad\jump{W_{\BH}}\Bn=\Bzr,\quad\jump{W_{\BH}}^{t}\Ba=\Bzr,\quad
  \jump{W}=\av{\lump{W_{\BH}},\jump{\BH}},
\end{equation}
where
\[
  \jump{\BH}=\BH_{+}-\BH_{-},\qquad\lump{W_{\BH}}=\hf(W_{\BH}(\BH_{+})+W_{\BH}(\BH_{-})).
\]
We have shown in \cite{grtrpcx} that the jump set for the energy
(\ref{ex:ener}) consists of the union of two hyperplanes
\[
\mathfrak{J}_{-}=\{\BH\in\bb{M}:\Trc\BH=\Gth_{1}\},\qquad
\mathfrak{J}_{+}=\{\BH\in\bb{M}:\Trc\BH=\Gth_{2}\},
\]
so that if $\Trc\BH_{-}=\Gth_{1}$, then the set of corresponding
$\BH_{+}$ is a projective plane worth of points
\[
\BH_{+}=\BH_{-}+(\Gth_{2}-\Gth_{1})\tns{\Bn},\quad|\Bn|=1.
\]
As we have already remarked, inequality (\ref{ex:keyineq}) implies that
$\mathfrak{J}_{\pm}\subset\CA$. 

 The lemma  formulated below,
which is valid for general energy densities $W(\BF)$, shows that
our results from \cite{grtrnc} imply that all matrices
$\BH_{t}=t\BH_{+}+(1-t)\BH_{-}$ would also fail to be rank-one convex for all
$t\in(0,1)$, proving equality (\ref{ex:binodal}), as the line segments
$\BH_{t}$ cover all points $\BH$, satisfying $\Gth_{1}<\Trc\BH<\Gth_{2}$.
\begin{lemma}
  \label{lem:r1line}
Suppose $\BH_{\pm}$ is a corresponding pair of points on the jump set $\mathfrak{J}$ of
the energy density function $W(\BH)$. Suppose that
\begin{enumerate}
\item[(i)] $W(\BH)\ge W(\BH_{-})+\av{\BP_{-},\BH-\BH_{-}}$
    for any matrix $\BH$, such that $\rank(\BH-\BH_{-})=1$, where
    $\BP_{-}=W_{\BH}(\BH_{-})$;
\item[(ii)] None of the matrices $\BH_{t}=t\BH_{+}+(1-t)\BH_{-}$, $t\in(0,1)$
  satisfies the jump set equations (\ref{js}) with $\BH_{-}$ as the
  corresponding pair.
\end{enumerate}
Then all of the matrices $\BH_{t}$, $t\in(0,1)$ must fail rank-one convexity
in the sense that they must fail to satisfy the inequality
\begin{equation}
  \label{Weiert}
  W(\BH)\ge W(\BH_{t})+\av{\BP_{t},\BH-\BH_{t}}\quad\forall\BH\colon\;\rank(\BH-\BH_{t})=1,
\end{equation}
where $\BP_{t}=W_{\BH}(\BH_{t})$.
\end{lemma}
\begin{proof}
  Suppose, with the goal of getting a contradiction, that (\ref{Weiert}) holds. Then, since matrices
  $\BH_{-}$, $\BH_{t}$ and $\BH_{+}$ all lie on the same rank-one line, we obtain
\[
W(\BH_{-})+t\av{\BP_{-},\jump{\BH}}\le W(\BH_{t})\le
W(\BH_{+})-(1-t)\av{\BP_{t},\jump{\BH}}.
\]
Now, using the jump set equations (\ref{js}), we obtain
\[
(1-t)(\av{\BP_{-},\jump{\BH}}-\av{\BP_{t},\jump{\BH}})\ge 0.
\]
Using the fact that $\BH_{t}-\BH_{-}=t\jump{\BH}$ we obtain
\begin{equation}
  \label{phii}
  \av{\BP_{t}-\BP_{-},\BH_{t}-\BH_{-}}\le 0.
\end{equation}
The variational significance of inequality (\ref{phii}), called the phase
interchange stability inequality, was understood in \cite{grtrnc}, where,
according to \cite[Lemma~4.1]{grtrnc}, the pair of rank-one related matrices
$\BH_{-}$ and $\BH_{t}$ would have to satisfy the jump set equations
(\ref{js}), in contradiction with the assumption (ii) of the lemma. Thus, the
assumption that (\ref{Weiert}) holds cannot be true, and all matrices
$\BH_{t}$, $t\in(0,1)$, must fail to be rank-one convex.
\end{proof}

\begin{remark} The energy density $QW_{0}(\BH)=RW_{0}(\BH)$ is non-convex as it has a
  double-well shape along the multiples of $\BI$, specifically,
  $QW_{0}(\Gve\BI)=\Phi^{**}(d\Gve)-d(d-1)\mu\Gve^{2}$. This is clear from the fact
  that its bulk modulus is negative in the interval
  $\Gth\in(\Gth_{1},\Gth_{2})$, where $\Gth_{1}$ and $\Gth_{2}$ are given by
  (\ref{th12}) for the bi-quadratic material model (\ref{linisel}). Moreover,
  the value of this modulus is constant and equal to $-2(d-1)\mu/d$,
  cf. \cite{buhala83}, which corresponds exactly to the threshold for the loss
  of strong ellipticity of the equilibrium equations (saturation of the
  Legendre-Hadamard conditions). This is associated with the generic
  degeneration of the acoustic tensor along rank one envelopes as we explain
  in our Appendix~\ref{sec:ac}.  The vanishing of the velocity of the
  longitudinal waves in our example turns the phase-transforming material into
  an elastic \ae ther \cite{erto56}.  
\end{remark}

We are now ready to show that for a class of energies $W(\BH)$ to which our example
belongs all strong local minimizers in a hard device (Dirichlet \bc s) must be global. The first observation, leading to this result is
that any two energies $W$ and $W'$ that have one and the same quasiconvex
envelope are equivalent as far as existence of metastable states are
concerned. This is because for any metastable state $\Bu$
\[
\int_{\GO}W(\Grad\Bu)d\Bx=\int_{\GO}QW(\Grad\Bu)d\Bx.
\]
Thus, any metastable configuration for $W$ is also metastable for $QW$. Theorem~\ref{th:exqcx} shows that our energy has the property
\begin{equation}
  \label{translcalss}
  QW(\BH)=N(\BH)+C(\BH),
\end{equation}
where $N(\BH)$ is a null-Lagrangian and $C(\BH)$ is convex and $C^{1}$ smooth.
\begin{theorem}
  \label{th:nometa}
Let $W(\BH)$ be the energy density satisfying (\ref{translcalss}). Then any Lipschitz strong local minimizer of $\int_{\GO}W(\Grad\Bu)d\Bx$ with prescribed Dirichlet \bc s is a global minimizer.
\end{theorem}
\begin{proof}  
Let $\Bv$ be a Lipschitz competitor
that agrees with $\Bu$ on $\dOm$. Then
\begin{equation}
  \label{cx}
  \int_{\GO}N(\Grad\Bv)d\Bx=\int_{\GO}N(\Grad\Bu)d\Bx,\qquad
C(\Grad\Bv)\ge C(\Grad\Bu)+\av{C_{\BF}(\Grad\Bu),\Grad(\Bv-\Bu)}.
\end{equation}
Since $\Bu$ is an equilibrium we have
\[
0=\Div W_{\BF}(\Grad\Bu)=\Div QW_{\BF}(\Grad\Bu)=\Div N_{\BF}(\Grad\Bu)+\Div C_{\BF}(\Grad\Bu)=\Div
C_{\BF}(\Grad\Bu).
\]
The equality of $W_{\BF}(\BH)=QW_{\BF}(\BH)$ for any $\BH\in\BCA$ follows from the fact that the function $W(\BH)-QW(\BH)$ is nonnegative, $C^{1}$ smooth and attains its minimum value of 0 at all $\BF\in\BCA$.
Now, the integration by parts in the second term on the \rhs\ in the inequality in (\ref{cx}) implies
$\int_{\GO}C(\Grad\Bv)d\Bx\ge \int_{\GO}C(\Grad\Bu)d\Bx$. It follows that
\[
\int_{\GO}QW(\Grad\Bv)d\Bx\ge\int_{\GO}QW(\Grad\Bu)d\Bx.
\]
\end{proof}
\begin{remark}
An important observation was made by Sivaloganathan and Spector in \cite{sisp18} that $C(\BH)$
does not need to be convex for inequality in (\ref{cx}) to hold. In fact, as the authors show, it is only
required that $C^{**}(\Grad\Bu(\Bx))=C(\Grad\Bu(\Bx))$ for all $\Bx\in\GO$,
i.e. $C(\BF)$ agrees with its convex hull at all values of $\Grad\Bu$.
\end{remark}

\section{Global minimizers}
\setcounter{equation}{0}
\label{sec:necsuf}
 In this section we temporarily switch to a more general
  setting.  For instance, we assume that  the ``deformation'' $\By$ is a vector field
  $\By:\GO\to\bb{R}^{m}$. In fact, the analysis presented below  does not  depend on either   
  dimension  of $\Bx$ space or the dimension of  $\By$ space. It also does not depend  on the specific form of $W(\BF)$.
  
Consider the problem of attainment in the definition of the quasiconvex
envelope
\begin{equation}
  \label{qcx}
  QW(\BF)=\inf_{\By\in\BF\Bx+W_{0}^{1,\infty}(\GO;\bb{R}^{m})}\dashint_{\GO}W(\Grad\By)d\Bx,
\end{equation}
where $\GO\subset\bb{R}^{d}$ is a star-shaped Lipschitz domain. Below we  show that Lipschitz equilibrium configurations for (\ref{qcx}) must be global minimizers, provided they satisfy the necessary conditions for metastability and are sufficiently regular near $\dOm$.

Recall first  the well-known necessary conditions for Lipschitz strong
local minima of variational functionals. 

The first classical necessary condition  is the
Euler-Lagrange equation
\begin{equation}
  \label{EL}
  \Div\BP(\Grad\By)=\Bzr,\qquad\BP(\BF)=W_{\BF}(\BF),
\end{equation}
which even weak local minimizers have to satisfy. The second, is
the Noether   equation
\begin{equation}
  \label{Noether}
  \Div\BP^{*}(\Grad\By)=\Bzr,\qquad \BP^{*}(\BF)=W(\BF)\BI_{d}-\BF^{t}\BP(\BF).
\end{equation}
We recall that if
$\By(\Bx)$ is of class $C^{2}$ then (\ref{Noether}) is a consequence
of (\ref{EL}), while  there are Lipschitz
configurations satisfying (\ref{EL}), but not (\ref{Noether}). For instance, the Maxwell relation, which is the last equation in (\ref{js}), is a consequence of  (\ref{Noether}), but not of (\ref{EL}).
\begin{definition}
  \label{def:station}
A configuration $\By\in W^{1,\infty}(\GO;\bb{R}^{m})$ is called stationary   
if it satisfies both (\ref{EL}) and (\ref{Noether}).
\end{definition}

The second  classical necessary condition is quasiconvexity $\Grad\By\in\BCA$ for a.e. $\Bx\in\GO$, where $\BCA$ is defined in (\ref{admset}). It was shown in \cite{grtrnc} that for configurations, whose only singularities are smooth phase boundaries the combination of (\ref{EL}) and quasiconvexity implies stationarity. In general, however, an example in \cite{krta03}  shows that there could be weak
local minimizers for strictly quasiconvex energies that are not strong local minimizers; most probably, the configuration in \cite{krta03} fails  stationarity.

It turns out that in the case affine \bc s in a star-shaped domain, the 
 two necessary conditions formulated above are sufficient for $\By(\Bx)$ to be a global
minimizer, provided we assume its regularity near the boundary of the domain.
\begin{theorem}
  \label{th:ELA}
Let $\GO\subset{\bb{R}^{d}}$ be a star-shaped domain with a Lipschitz boundary. Assume that
\begin{itemize}
\item[(i)] $\By:\GO\to\bb{R}^{m}$ is Lipschitz continuous,
\item[(ii)] $\By$ solves (\ref{EL}) and (\ref{Noether}) in the sense of distributions,
\item[(iii)] $\Grad\By(\Bx)$ is locally stable for a.e. $\Bx\in\GO$,
\item[(iv)] $\By(\Bx)=\BF\Bx$ for all $\Bx\in\dOm$,
\item[(v)] There exists $\Ge>0$, such that $\By\in C^{1}(\bra{\GO_{\Ge}};\bb{R}^{m})$, 
where $\GO_{\Ge}=\{\Bx\in\GO:{\rm dist}(\Bx,\dOm)<\Ge\}$. 
\end{itemize}
Then $\By(\Bx)$ is the global minimizer in (\ref{qcx}), i.e.
\begin{equation}
  \label{QW}
  \int_{\GO}W(\Grad\By)d\Bx=|\GO|QW(\BF).
\end{equation}
\end{theorem}
\begin{proof}
  While each ingredient of the   proof presented below
  was in one  way or another already present in the  uniqueness proof of Knops and Stuart
  \cite{knst84},  our theorem is new, since it does not assume that the energy
  density $W(\BF)$ is quasiconvex.
The idea is to express
  the energy of an  equilibrium configuration as a boundary integral using
  a far reaching generalization of a formula in \cite{knst84} which,
  by itself,  can be viewed as   a nontrivial  nonlinear generalization of the
  classical Clapeyron theorem\footnote{The classical  Clapeyron theorem in linear
    elasticity states that the elastic energy
    $E[\Bu]=(1/2)\int_{\GO}\av{\SFC(\Bx)e(\Bu),e(\Bu)}d\Bx$ of an equilibrium configuration
     can be computed by the formula 
    $E[\Bu]=(1/2)\int_{\dOm}\BGs\Bn\cdot\Bu dS$, where $\BGs=\SFC(\Bx)e(\Bu)$
    is the stress tensor, \cite{fotr03}.  
    } 
  \begin{lemma}
    \label{lem:Clapeyron}
Suppose $\By(\Bx)$ is a Lipschitz stationary equilibrium in a Lipschitz domain
$\GO$. Then 
\begin{equation}
  \label{Clapeyron}
  E[\By]=\int_{\GO}W(\Grad\By)=\nth{d}\int_{\dOm}\{
\BP\Bn\cdot\By+\BP^{*}\Bn\cdot\Bx\}d\Bx,
\end{equation}
where $\BP\Bn$ and $\BP^{*}\Bn$ can be regarded as trace
functionals, since they act on Lipschitz functions $\By(\Bx)$ and $\Bx$, respectively.
  \end{lemma}
  \begin{proof}
    \[
\int_{\dOm}\BP^{*}\Bn\cdot\Bx dS=\int_{\GO}\av{\BP^{*},\Grad\Bx}d\Bx=
\int_{\GO}\Trc(\BP^{*})d\Bx=d\int_{\GO}W(\Grad\By)d\Bx-\int_{\GO}\av{\BP,\Grad\By}d\Bx.
\]
Equation (\ref{Clapeyron}) follows from
\[
\int_{\GO}\av{\BP,\Grad\By}d\Bx=\int_{\dOm}\BP\Bn\cdot\By dS.
\]
  \end{proof}
We now apply formula (\ref{Clapeyron}) for the energy of $\By(\Bx)$. Using
the \bc\ (iv) and the assumption (v) we have at all $\Bx\in\dOm$
\[
\Grad\By=\BF+\Ba\otimes\Bn,\qquad \Ba=\dif{\By}{\Bn}-\BF\Bn.
\]
We therefore compute
\[
dE[\By]=\int_{\dOm}(W(\BF+\Ba\otimes\Bn)-\BP(\BF+\Ba\otimes\Bn)\Bn\cdot\Ba)(\Bn\cdot\Bx)dS
\]
next, we use the local stability condition (iii) and appeal to \cite[Lemma~4.2]{grtrnc}.
We quote the part of the lemma we need for the sake of completeness.
\begin{lemma}
  \label{lem:r1cx}
Let $V(\BF)$ be a rank-one convex function such that $V(\BF)\le W(\BF)$. Let
\[
\CA_{V}=\{\BF\in\CO:W(\BF)=V(\BF)\},
\]
where $\CO$ is an open subset of $\bb{M}$ on which $W(\BF)$ is of class
$C^{1}$.  Then for every $\BF\in\CA_{V}$, $\Bu\in\bb{R}^{m}$, and $\Bv\in\bb{R}^{d}$
\begin{equation}
  \label{coolineq}
  V(\BF+\Bu\otimes\Bv)\ge W(\BF)+\BP(\BF)\Bv\cdot\Bu.
\end{equation}
\end{lemma}
We now apply Lemma~\ref{lem:r1cx} by choosing $V(\BF)$ to be $QW(\BF)$. By
assumption (iii) for each $\Bx\in\dOm$ the field
$\Grad\By(\Bx)=\BF+\Ba\otimes\Bn$ is in $\CA_{V}$. Choosing
$\Bu\otimes\Bv=-\Ba\otimes\Bn$, inequality (\ref{coolineq}) becomes
\begin{equation}
  \label{QWineq}
  QW(\BF)\ge W(\BF+\Ba\otimes\Bn)-\BP(\BF+\Ba\otimes\Bn)\Bn\cdot\Ba.
\end{equation}
Finally, we use the assumption of $\GO$ being star-shaped. If we choose the
origin at the star point, then the function $\Bn(\Bx)\cdot\Bx$ is always
non-negative at all points on $\dOm$. Therefore, inequality (\ref{QWineq})
implies
\[
dE[\By]\le QW(\BF)\int_{\dOm}(\Bn\cdot\Bx)dS=d|\GO|QW(\BF).
\]
Since $|\GO|QW(\BF)$ is the minimal value of $E[\By]$ we conclude that
$E[\By]=QW(\BF)$, and $\By(\Bx)$ is the global minimizer.
\end{proof}

\section{Multiplicity of global minimizers}
\setcounter{equation}{0}
\label{sec:attainment}
 The absence of strong local minimizers in hard device with
  affine \bc s demonstrated in the previous section leaves
the question of the characterization of all global minimizers in
\begin{equation}
  \label{globmin}
  QW_{0}(\BH_{0})=\inf_{\Bu|_{\dOm}=\BH_{0}\Bx}\dashint_{\GO}W_{0}(\Grad\Bu)d\Bx,
\end{equation}
where $W_{0}(\BH)$ is defined in (\ref{ex:ener}).
We remark that the construction of the periodic laminate in
Section~\ref{sec:slm} provides us with a minimizing sequence in (\ref{globmin}) in
\emph{arbitrary domains}. In this section we raise the question of
attainability of the minimum in (\ref{globmin}). In this regard, laminates can no
longer be used as the \bc s 
\begin{equation}
  \label{affbc}
  \Bu(\Bx)=\BH_{0}\Bx,\quad\Bx\in\dOm.
\end{equation}
are inconsistent with $\Grad\Bu$ taking exactly two specific values almost everywhere
in $\GO$. We will therefore attempt to find energy minimizers using the
sufficiency conditions from Theorem~\ref{th:ELA}.  Accordingly, we only need
to find solutions $\Bu(\Bx)$ of (\ref{EL}) that satisfies (\ref{affbc}) while
respecting the stability condition
\begin{equation}
\Div\Bu(\Bx)\not\in(\Gth_1,\Gth_2)\mbox{ for a.e. }\Bx\in\GO.
\label{ex:bin}
\end{equation}
Let $\BGf(\Bx)=\Bu(\Bx)-\BH_{0}\Bx$. Then (\ref{EL}) can be written in
terms of $\BGf$ as follows
\begin{equation}
\mu\GD \BGf+\nabla\Phi'(\Div\Bu) = \mu \Grad(\Div\BGf),\quad \Bx\in \GO.
\label{zequil}
\end{equation}
Integrating the dot product of $\BGf(\Bx)$ and (\ref{zequil}) by parts
we obtain
\[
\int_\GO\left\{\mu(|\Div\BGf|^2-|\Grad\BGf|^2)-\Phi'(\Div\Bu)\Div\BGf\right\}d\Bx=0.
\]
Now we use the algebraic identity
\[
|\Div\BGf|^2-|\Grad\BGf|^2=2J_2(\Grad\BGf)-\hf|\Grad\BGf-(\Grad\BGf)^{t}|^2
\]
and the fact that $J_2(\Grad\BGf)$ given by (\ref{J2inv}) is a null-Lagrangian to
obtain
\begin{equation}
\int_\GO\left\{\frac{\mu}{2}|\Grad\BGf-(\Grad\BGf)^{t}|^2+\Phi'(\Div\Bu)\Div\BGf\right\}d\Bx=0.
\label{LMtrick}
\end{equation}
Now let us show that (\ref{ex:bin}) implies that
\begin{equation}
\int_\GO\Phi'(\Div\Bu)\Div\BGf(\Bx) d\Bx\ge 0.
\label{keyineq}
\end{equation}
Indeed, from the geometric interpretation of $\Gth_1$ and $\Gth_2$
(as the points of common tangency 
) we conclude that in view of (\ref{ex:bin}) we can write 
\begin{equation}
\Phi(\Div\Bu(\Bx))=\Phi^{**}(\Div\Bu(\Bx))
\label{helpful}
\end{equation}
for a.e. $\Bx\in\GO$.
For a convex function $\Phi^{**}(\Gth)$ the tangent line to the graph
of that function lies always below the graph which means that 
\[
\Phi^{**}(\eta)\ge \Phi'(\Gth)(\eta-\Gth)+\Phi(\Gth)
\]
for any $\eta\in\mathbb{R}$ and any $\Gth\not\in(\Gth_{1},\Gth_{2})$.
Substituting $\eta=\Trc\BH_{0}$ and $\Gth=\Div\Bu(\Bx)$ we obtain
\[
\Phi'(\Div\Bu(\Bx))\Div\BGf(\Bx)\ge\Phi(\Div\Bu(\Bx))-\Phi^{**}(\Trc\BH_{0})=
\Phi^{**}(\Div\Bu(\Bx))-\Phi^{**}(\Trc\BH_{0})
\]
for a.e. $\Bx\in\GO$. Now   (\ref{keyineq}) follows from Jensen's
inequality for a convex function $\Phi^{**}(\Gth)$ and the fact that 
\begin{equation}
\label{thav}
\nth{|\GO|}\int_{\GO}\Div\Bu(\Bx)d\Bx=\Trc\BH_{0}.
\end{equation}
Now, inequality (\ref{keyineq}) together with  (\ref{LMtrick}) imply
that $\Grad\BGf=(\Grad\BGf)^{t}$. Therefore,
\begin{equation}
  \label{divcurl}
  0=\Div(\Grad\BGf-(\Grad\BGf)^{t})=\GD\BGf-\Grad(\Div\BGf).
\end{equation}
Thus, (\ref{zequil}) implies that $\Grad\Phi'(\Div\Bu)=0$, and there exists a
constant $P_{0}$, such that
\begin{equation}
  \label{intEL}
  \Phi'(\Div\Bu) = P_{0}.
\end{equation}
The constraint (\ref{ex:bin}) then implies that either
$\Div\Bu=\Gth_{0}$ is a constant function in $\GO$, or that $P_{0}=\Phi'(\Gth_{1})=\Phi'(\Gth_{2})$, 
in which case there exists a subset
$A$ of $\GO$ such that 
\begin{equation}
\Div\Bu(\Bx)=\Gth_1\chi_A(\Bx)+\Gth_2\chi_{\GO\setminus A}(\Bx),
\label{theta}
\end{equation}
where $\Gth_{1}<\Gth_{2}$ are the endpoints of the interval
$\{\Gth\in\bb{R}:\Phi^{**}(\Gth)<\Phi(\Gth)\}$ (see Fig.~\ref{fig:Phitheta}).
Recalling (\ref{thav}) we conclude that we must have $\Div\Bu=\Trc\BH_{0}$, if
$\Trc\BH_{0}\not\in(\Gth_{1},\Gth_{2})$. But then $\Div\BGf=0$, and thus,
according to (\ref{divcurl}), $\BGf$ solves $\GD\BGf=0$ in $\GO$, with
$\BGf=0$ on $\dOm$. It follows that $\BGf=0$ and we conclude that
$\Bu(\Bx)=\BH_{0}\Bx$ is the only equilibrium satisfying (\ref{ex:bin}).

If $\Trc\BH_{0}\in(\Gth_{1},\Gth_{2})$, then in addition to
$\Bu(\Bx)=\BH_{0}\Bx$ we may have other solutions satisfying (\ref{theta}),
in which case
\begin{equation}
  \label{vf}
  \Go=\frac{|A|}{|\GO|}=\frac{\Trc\BH_{0}-\Gth_{2}}{\Gth_{1}-\Gth_{2}}
\end{equation}
is the volume fraction of the phase in which $\Div\Bu(\Bx)=\Gth_{1}$.

If, in addition, $\GO$ is simply connected, then there exists a scalar
potential $h\in W_{0}^{2,2}(\GO)$, such that $\BGf=\Grad h$. Then,
in terms of the potential $h$ we obtain a  free boundary problem
\begin{equation}
\left\{
\begin{array}{l}
\GD h=(\Gth_1-\Gth_2)(\chi_A(\Bx)-\Go),\quad \Bx\in \GO\\[1ex]
h\in W_{0}^{2,2}(\GO).
\end{array}
\right.
\label{inverse}
\end{equation}
As formulas (\ref{vf}), (\ref{inverse}) indicate, the volume fraction $\Go$ of the precipitate is
uniquely determined by the hard device loading $\BH_{0}$, whose shear
component has no effect on the precipitate morphology. 
We remark that while (\ref{inverse}) looks like a Cauchy problem for the
Poisson equation, the \rhs\ in (\ref{inverse}) is not fixed and the problem
would be solved if we can find just the right shape $A$ of the inclusion.

\begin{figure}[t]
  \centering
  \includegraphics[scale=0.45]{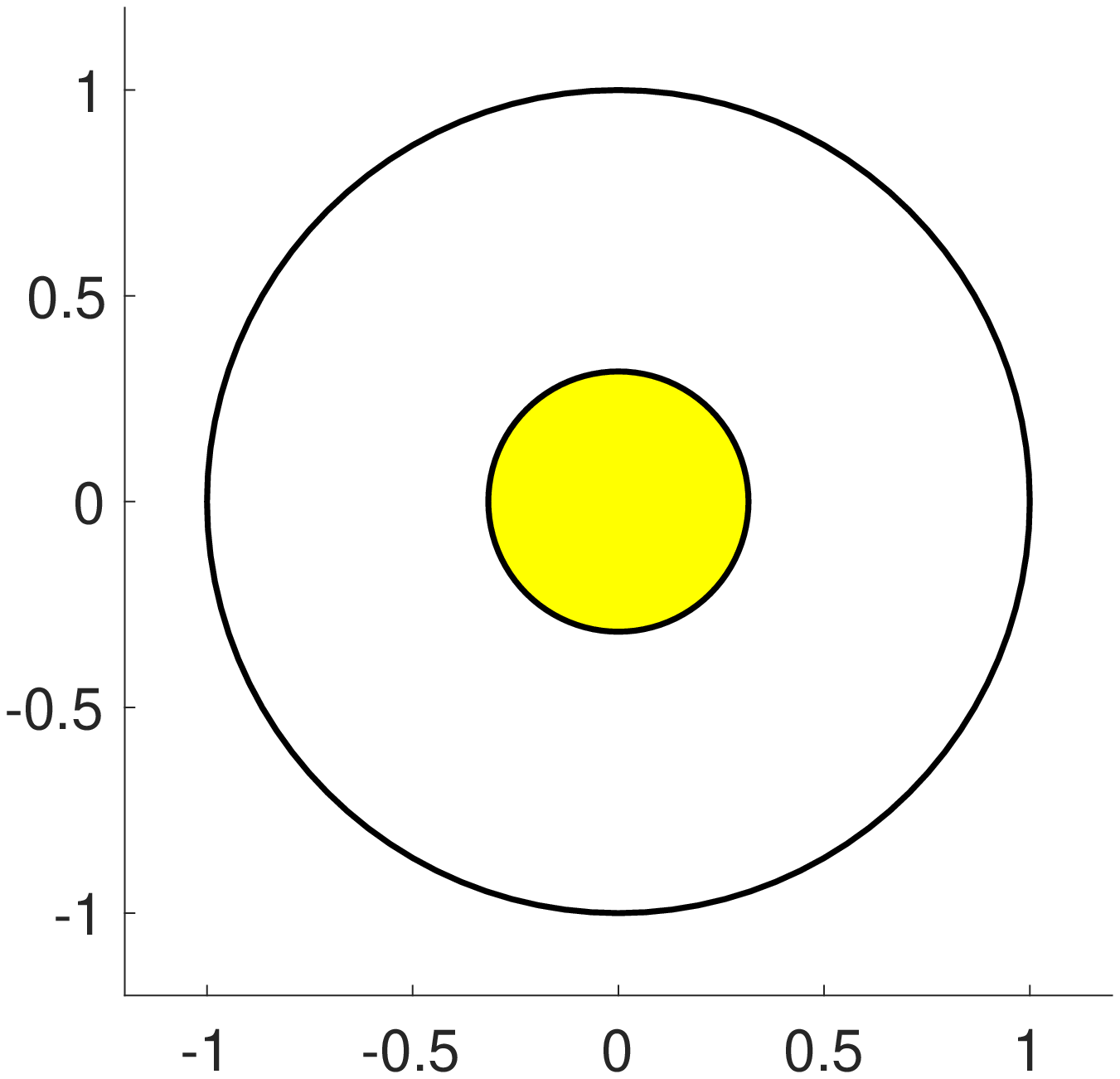}~~~~~~~~~~
\includegraphics[scale=0.45]{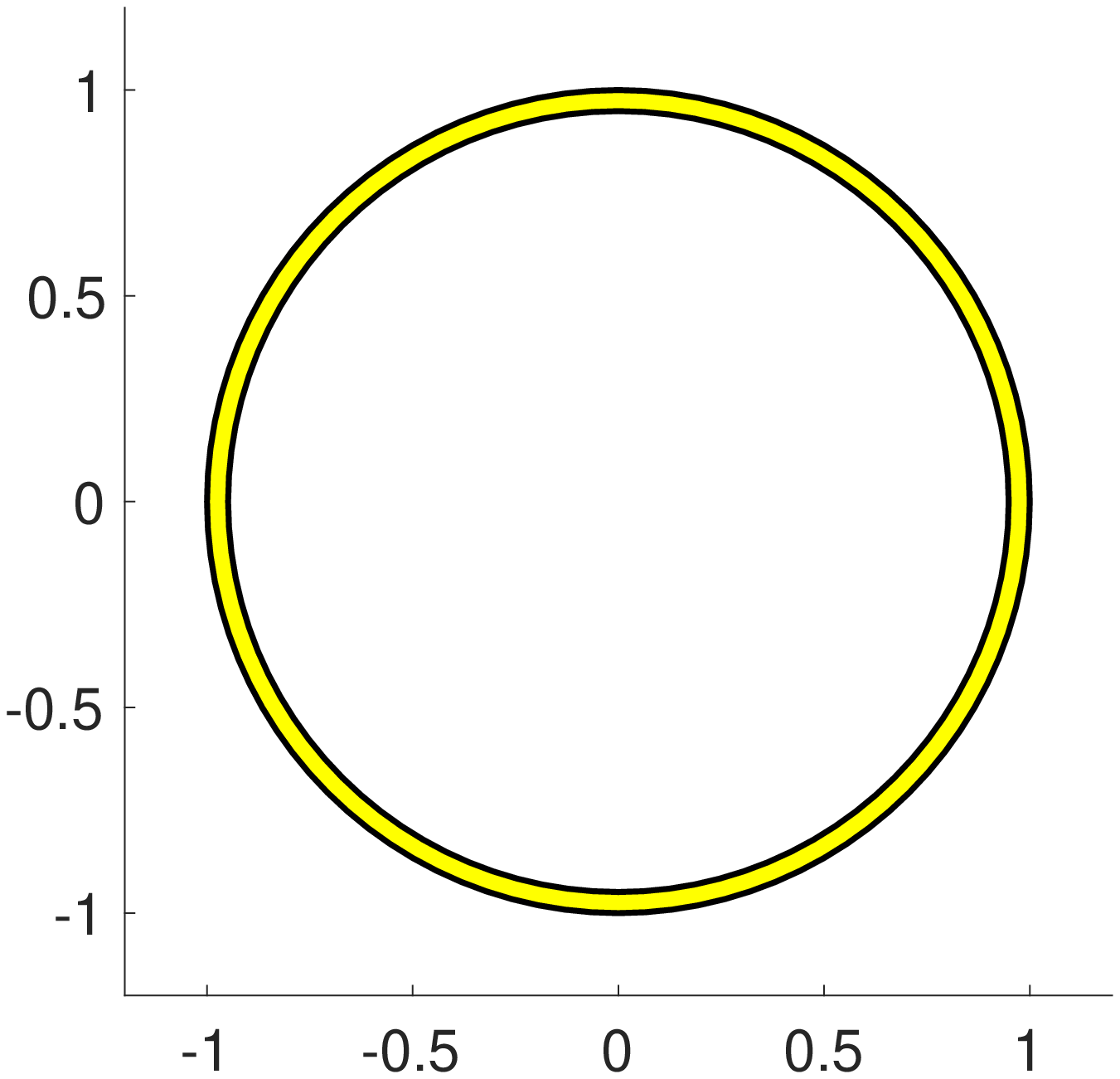}
  \caption{Morphological hysteresis showing two different morphologies of phase
    arrangements corresponding to 10\% volume fraction of phase 2 ($\Go=0.1$), shown in
    yellow. In the left panel nucleation starts at the center. In the right
    panel nucleation starts at the boundary.}
  \label{fig:mhyst}
\end{figure}
In fact, a solution of \eqref{inverse} can be easily found in explicit form for some
particularly simple domain geometries. For instance, it is easy to construct a
radially symmetric solution when $\GO$ is the unit ball, taking as the set $A$
to be either the concentric ball of radius $r_{0}=\Go^{1/d}$, or its
complement, in which case $r_{0}=(1-\Go)^{1/d}$. Then (\ref{inverse}) is
solved by $h(\Bx)=h_{\Go}(|\Bx|)$, given by
\begin{equation}
h_{\Go}(r)=\left\{\begin{array}{ll}
\vspace{1ex}
\nth{2d}(\Gth_{1}-\Gth_{2})(1-\Go)\left(r^{2}-
\frac{d(\Go^{2/d}-\Go)}{(d-2)(1-\Go)}\right), &\quad 0\le r\le \Go^{1/d}\\[2ex]
-\nth{2d}(\Gth_{1}-\Gth_{2})\Go\left(r^{2}+\frac{2}{(d-2)r^{d-2}}-\frac{d}{d-2}\right),
&\quad \Go^{1/d}\le r\le 1,
\end{array}\right.
\label{rotinvsol}
\end{equation}
in the former case. In the latter case, $h(\Bx)$ is given by the same
formula (\ref{rotinvsol}) where $\Gth_{1}$ and $\Gth_{2}$ are interchanged and
$\Go$ is replaced by $1-\Go$, i.e. $h(\Bx)=-h_{1-\Go}(|\Bx|)$. These two optimal configurations with the same energy are illustrated in Fig. \ref{fig:mhyst}. With the increase of the volume fraction $\Go$ the nucleus in the left panel (colored in yellow) will grow outward while the nucleus in the right panel (also colored in yellow)  will grow inward. At  $\Go=1$  both transformation   terminate as the second phase   completely takes over the whole domain. Note that the associated morphological mechanisms of a phase transformation, which start and end in the same configurations while going through energetically equivalent states, are nevertheless different.

One can show that the ``concentric sphere" solution (\ref{rotinvsol}) can be used
to construct an infinite family of solutions in any other Lipschitz
domain. This can be  done using  the  so-called Hashin's ``concentric sphere''
construction, \cite{hash62}. In this construction the set $A$ is a
countable union of variously scaled copies of the radially symmetric solution
(\ref{rotinvsol}) filling $\GO$ up to a set of Lebesgue measure zero. Let
$B(\Bx_{i},a_{i})\subset B(\Bx_{i},R_{i})$, $i=1,2,\ldots$ are the concentric
balls used in Hashin's construction, where
\[
\frac{a_{i}^{d}}{R_{i}^{d}}=\Go.
\]
The inner balls $B(\Bx_{i},a_{i})$ in the construction belong
to the set $A$, while the spherical shells
$B(\Bx_{i},R_{i})\setminus B(\Bx_{i},a_{i})$ belong to $\GO\setminus A$.
The function $h(\Bx)$
restricted to the  ball $B(\Bx_{i},R_{i})$
is given by
\[
h(\Bx)=R_{i}^{2}\:h_{\Go}\left(\frac{|\Bx-\Bx_{i}|}{R_{i}}\right),
\quad\Bx\in B(\Bx_{i},R_{i}),
\]
where $h_{\Go}$ is given by (\ref{rotinvsol}).
The function $h(\Bx)$ defined like this on each of the concentric balls
does indeed solve (\ref{inverse}), since on the boundary of each ball
$B(\Bx_{i},R_{i})$ both $h(\Bx)$ and $\Grad h(\Bx)$ are zero.

It is clear that the  ``concentric sphere'' construction in a non-circular domain will have  an infinite surface area (for the  formal proof
see \cite{merg54tr,wesl60}), which is  unacceptable in the physically more realistic context where phase boundaries (surfaces of gradient discontinuity) carry  additional surface energy. For this reason one may be interested whether other optimal microstructures  with  finite surface area exist as well. In some cases, when the domain $\GO$ is sufficiently simple, such low surface energy alternatives are indeed  known to exist.
 For example, if $\GO$ is an ellipsoid, then the low-surface-area minimizer $A$ is a confocal ellipsoid (or
its complement) \cite{berg,mi80,trtr,zhik91,gra}. However, if $\GO$ is  more complex, for instance, a
square,  the   low-surface-area minimizer, if it exists, may be  nontrivial. To corroborate this claim,  we show in  Appendix~\ref{sec:MOMA} that in this case  an    optimal inclusion with square symmetry and simple topology simply does not exist.

\section{Conclusions}
\setcounter{equation}{0}
\label{sec:conc}

In this paper  we considered a classical problem of nonlinear  elasticity for a material undergoing phase transition. Mathematically this physical problem  reduces to  a non-convex vectorial problem of the calculus of variations. 

In this framework we presented 
  a simple, yet non-trivial example of an energy density (of a material) for
  which 
     the absence of metastability  in a hard device coexists with a wild nonuniqueness of global minimizers.

 Metastability in elastostatics, understood as the existence  of strong local
  minimizers that are not global, is ubiquitous as    hysteresis is a typical phenomenon accompanying  martensitic phase  transitions. While for the Neumann \bc s the existence of such local minimizers is indeed  common, as its existence  can be  linked to the generic  incompatibility of the energy wells, here we showed  that for the Dirichlet \bc s the  incompatibility of the  wells
 does not necessarily cause 
``metastability". 
More specifically, we have presented   an example of
 the energy density of a  hyperelastic material with non-rank-one convex,
 double-well energy and non-rank one connected energy wells, for which we could prove  the lack of strong local
 minimizers which are not global on any domain and for any Dirichlet \bc
 s. The analytical transparency of our arguments was due to the utmost simplicity of the chosen energy density which is geometrically linear and  isotropic.
  
 For the same material model we could fully characterize the necessary and
 sufficient conditions defining global energy minima in hard device loading.  The obtained  conditions   allowed us to reveal  the multiplicity of global minimizers with the associated nonuniqueness unrelated to either objectivity or crystallographic  symmetry. As an  important element of this analysis we used a novel way of expressing the  energy of equilibrium configurations as boundary integrals which can be  viewed as a nontrivial nonlinear generalization of the classical Clapeyron  theorem.

While we showed that the relaxation (quasi-convexification) of the energy  in our model can be achieved by simple lamination or coated sphere construction, these optimal microstructures  are hardly physical as the associated surface area is infinite.  In  real physical  situations finite surface energy plays the role of the selection mechanism among otherwise energetically equivalent configurations
 \cite{Kloucek:1994:CMM,komu,Dolzmann:1995:ISE}, ruling out  constructions with infinite surface area. Therefore, of particular interest in physical applications are global minimizers with finite surface area, and we  explicitly computed a one-parameter
  family of non-affine energy minimizing configurations for the case of a
  finite domain with smooth  boundary. An interesting aspect of the obtained solution is that,  despite the absence of metastability  and the associated constitutive  hysteresis, it shows that the  direct and the reverse transformation  may follow different morphological paths while traversing energetically equivalent configurations.   The possibility of such a morphological hysteresis may be found advantageous  in applications if the ways of  manipulating  the microstructure,    beyond the scale and range of  classical continuum elasticity,  are employed.  
 
\medskip

\textbf{Acknowledgments.} YG is grateful to Pavel Etingof who pointed out the
references for the proof of the infinity of the surface area of Hashin's
concentric sphere construction in 1992.  

\textbf{Funding.} YG was supported by the National Science Foundation
under Grant No. DMS-2005538. The work of LT was supported by the French grant
ANR-10-IDEX-0001-02 PSL. 

\textbf{Author contributions.}  Both authors wrote the main manuscript
text and reviewed the manuscript.

\section{Declarations}
\textbf{Competing interests.} The authors declare no competing interests.

\appendix

\section{The degeneracy of acoustic tensors of rank-one envelopes}
\setcounter{equation}{0}
\label{sec:ac}
Here we prove that the acoustic tensor of a rank-one convex envelope $RW$ of
the non rank-one convex energy must have a degenerate direction at all points
$\BF_{0}$ where $RW(\BF_{0})<W(\BF_{0})$. We also show that in particularly simple
situations, such as the one discussed in this paper, this
property may be even sufficient to compute the whole rank-one convex envelope.

We recall that the acoustic tensor of the energy $W(\BF)$ at $\BF=\BF_{0}$ in the direction $\Bn$ is a quadratic form $\BA(\Bn)$ defined by
\[
\BA(\Bn)\Ba\cdot\Ba=\av{W_{\BF\BF}(\BF_{0})(\Ba\otimes\Bn), \Ba\otimes\Bn}.
\]

\begin{theorem}
\label{th:acou}
Let $\BF_0$ be fixed and suppose $RW(\BF_0)<W(\BF_0)$, where $RW$
denotes the rank-1 convexification of $W$.
Assume further that $W$ and $RW$ are $C^2$ near $\BF_0$. Let $\BA_{0}(\Bn)$ be the acoustic
tensor of $RW$ at $\BF_{0}$.
Then there is a direction $\Bn$
such that
$ 
\det\BA_{0}(\Bn)=0.
$ 

\end{theorem}

\begin{proof}
Assume that there is no such direction $\Bn$. Since $RW$ is necessarily rank-1
convex we can conclude that for any direction $\Bn$ the matrix $\BA_{0}(\Bn)$
is positive semidefinite. Since the function $\Bn\mapsto\BA_{0}(\Bn)$ is continuous
there exists a positive number $\Ga$ such that for every unit vector $\Bn$
$$
\BA_{0}(\Bn)\ge \Ga \BI.
$$

Since $RW\in C^2$ near $\BF_0$ there exists a number $\Gd>0$ such that
for every $\BF$ satisfying $|\BF-\BF_0|<\Gd$ the following inequalities hold:
\begin{enumerate}
\item $\BA(\Bn)\ge \hf\Ga \BI,$
where $\BA(\Bn)$ is the acoustic tensor of $RW$ at $\BF$.
\item $W(\BF)-RW(\BF)\ge\hf(W(\BF_0)-RW(\BF_0))=\Gb>0.$
\end{enumerate}
Now consider a function
$$
T_{\Ge}(\BF)=RW(\BF)+\Ge\phi(\frac{\BF-\BF_0}{\sqrt[4]{\Ge}}),
$$
where $\phi$ is a smooth nonnegative function supported on the unit ball in
$\bb{M}$ and such that $\phi(\Bzr)=1$.
We can choose $\Ge$ so small that the following inequalities
hold:
\begin{enumerate}
\item $\|\Ge\phi\|_{L^{\infty}}<\hf\Gb,$
\item $\|\sqrt{\Ge}\nabla_{\BF}\nabla_{\BF}\phi\|_{L^{\infty}}\le \nth{4}\Ga,$
\item supp $\phi(\frac{\BF-\BF_0}{\sqrt[4]{\Ge}})\subset B(\BF_0,\Gd)$,
where $B(\BF_0,\Gd)$ is the ball of radius $\Gd$ around $\BF_0$ in $\bb{M}$.
\end{enumerate}
Then $T_{\Ge}(\BF)\le W(\BF)$ and the acoustic tensor $\BA_{\Ge}(\Bn)$ of
$T_{\Ge}(\BF)$ satisfies
$$
\BA_{\Ge}(\Bn)\ge \nth{4}\Ga\BI
$$
in the sense of quadratic forms.
Thus $T_{\Ge}(\BF)$ is rank-1 convex and $T_{\Ge}(\BF)\le W(\BF)$ but
$T_{\Ge}(\BF_{0})>RW(\BF_{0})$.
Contradiction. Thus our assumption is false and at
every point $\BF_0$ where $RW(\BF_{0})<W(\BF_{0})$ there is a direction
$\Bn$ such that the acoustic
tensor $\BA_{0}(\Bn)$ is degenerate.
\end{proof}
\begin{remark}
It follows that materials that transform by forming microstructures with sharp phase boundaries to accommodate deformations produced by the propagation of sound waves will
have a direction with a zero sound speed.
\end{remark}
We can apply Theorem~\ref{th:acou} to the energy (\ref{ex:ener}). Since our material is isotropic, the degeneration of the acoustic tensor may be either through $\mu=0$ or  through $\lambda +2 \mu=0$ with the latter also meaning that the bulk modulus $\Gk=-2\mu(d-1)/d$. The possibility that $\mu=0$ is excluded because the tangential shear modulus is the same at every deformation.

Suppose that we have somehow guessed that if
\[
W_{0}(\BH)=f(\Trc\BGve)+\mu|\dev{\BGve}|^{2},
\]
then 
\[
RW_{0}(\BH)=F(\Trc\BGve)+\mu|\dev{\BGve}|^{2}
\]
for some function $F$, yet to be determined.
In that case Theorem~\ref{th:acou} will let us determine the function 
$F(\Gth)$. It is easy to compute
that for any unit vector $\Bn$ and any
vector $\Ba\in\bb{R}^{d}$
\[
\av{RW_{0,\BH\BH}(\BH)(\Bn\otimes\Ba),\Bn\otimes\Ba}=
F''(\Trc\BGve)(\Ba,\Bn)^{2}+2\mu\left|\hf(\Bn\otimes\Ba+\Ba\otimes\Bn)-
\nth{d}(\Ba,\Bn)\BI\right|^{2}.
\]
Therefore, we get a formula for the acoustic tensor $\BA(\Bn)$.
\begin{equation}
\BA(\Bn)=\left(F''(\Trc\BH)+\frac{d-2}{d}\mu\right)\tns{\Bn}+\mu\BI.
\label{acou}
\end{equation}
We see that $\det\BA(\Bn)=\mu^{2}(F''(\Trc\BGve)+2(d-1)\mu/d)$ for all
directions $\Bn$. Thus, for all $\BH$ for which $RW_{0}(\BH)<W_{0}(\BH)$ we get
$F''(\Trc\BH)=-2(d-1)\mu/d$. The continuity of $RW_{0,\BH}$ implies that
at the boundary points $\Gth_{1}$ and $\Gth_{2}$ of the binodal region
we have $F'(\Gth_{1})=f'(\Gth_{1})$ and
$F'(\Gth_{2})=f'(\Gth_{2})$. Therefore, in the binodal region, where
$F'(\Gth)=-2(d-1)\Gth\mu/d+C$ for some constant $C$,
we must have $C=\Phi'(\Gth_{1})=\Phi'(\Gth_{2})$. Further, the continuity of $W_{0}$
implies that the affine function\footnote{The function $y(\Gth)$ is affine
  because the property of $F(\Gth)$ can be written as $y''(\Gth)=0$.} 
$y(\Gth)=F(\Gth)+(d-1)\mu\Gth^{2}/d$ would be  the
equation of the common tangent
to the graph of $\Phi(\Gth)$. Thus, we obtain that
$F(\Gth)+(d-1)\mu\Gth^{2}/d=\Phi^{**}(\Gth)$, and we recover the 
rank-one convex envelope of $W_{0}(\BH)$, which in this case is seen to coincide with its quasiconvexification (\ref{ex:qcx}).

\section{A non-existence of a topologically simple square symmetric minimizer in a square}
\setcounter{equation}{0}
\label{sec:MOMA}
We  begin with a general observation that if $A\subset\GO\subset\bb{R}^{2}$
is an open subset for which the problem (\ref{inverse}) has a solution $h(x,y)$, then
functions
\[
\Tld{h}_{+}(x,y)=h(x,y)+\nth{4}(\Gth_{1}-\Gth_{2})\Go(x^{2}+y^{2}),\quad
\Tld{h}_{-}(x,y)=h(x,y)-\nth{4}(\Gth_{1}-\Gth_{2})(1-\Go)(x^{2}+y^{2})
\]
are harmonic in $\GO\setminus A$ and $A$ respectively. We can then conclude  that the
functions $\Md\Tld{h}_{\pm}/\Md x - i \Md\Tld{h}_{\pm}/\Md y$ are analytic
in the complex variable $z=x+iy$ on their respective domains. Thus, the functions
\[
H_{+}(z)=\frac{2}{\Gth_{1}-\Gth_{2}}\left(\dif{h}{x}-i\dif{h}{y}\right)+
\Go\bra{z},\quad z\in \GO\setminus A,
\]
\[
H_{-}(z)=\frac{2}{\Gth_{1}-\Gth_{2}}\left(\dif{h}{x}-i\dif{h}{y}\right)-
(1-\Go)\bra{z}, \quad z\in A
\]
are also analytic. The boundary conditions in (\ref{inverse}) and the continuity of
$\Grad h$ across $\Md A$, representing the kinematic compatibility of the
displacement, 
imply that
\begin{equation}
  \label{complex}
\left\{
  \begin{array}{cc}
H_{+}(z)=\Go\bra{z},\quad & z\in\dOm,\\[1ex]
H_{+}(z)-H_{-}(z)=\bra{z},\quad & z\in\Md A.
  \end{array}
\right.
\end{equation}
Thus, the problem of finding global minimizers in $\GO$ with affine \bc s
reduces to the problem (\ref{complex}) in the theory
of complex analytic functions of one complex variable.

We now attempt to solve this problem when $\GO$ is a square
centered at the origin with diagonal of length 2. Then along the bottom
side of the square we have
$H_{+}(x-i/\sqrt{2})=\Go(x+i/\sqrt{2})=\Go(z+i\sqrt{2})$.
Therefore $H_{+}(z)=\Go(z+i\sqrt{2})$, while along the right side
of the square we have
$H_{+}(1/\sqrt{2}+iy)=\Go(1/\sqrt{2}-iy)=\Go(\sqrt{2}-z)$.
Thus, $H_{+}(z)=\Go(\sqrt{2}-z)$. These contradictory expressions for $H_{+}(z)$
can only be reconciled by a structure with topology indicated in
Figure~\ref{fig:square}, where the set $A$, in which $\Div\Bu=\Gth_{1}$, is shaded in yellow.
\begin{figure}
  \centering
\includegraphics[scale=0.3]{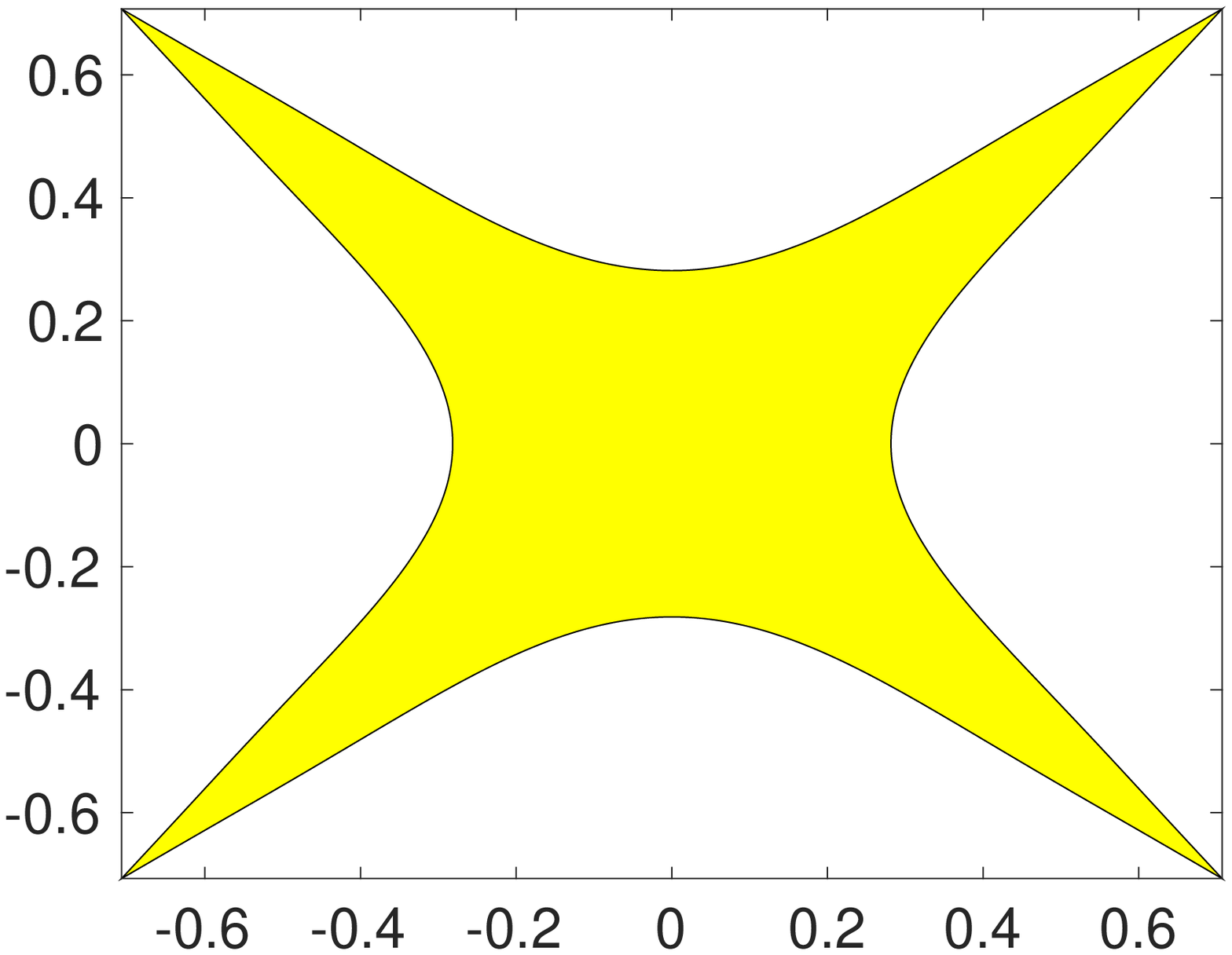}~~~~~~~~~~~~~
\includegraphics[scale=0.3]{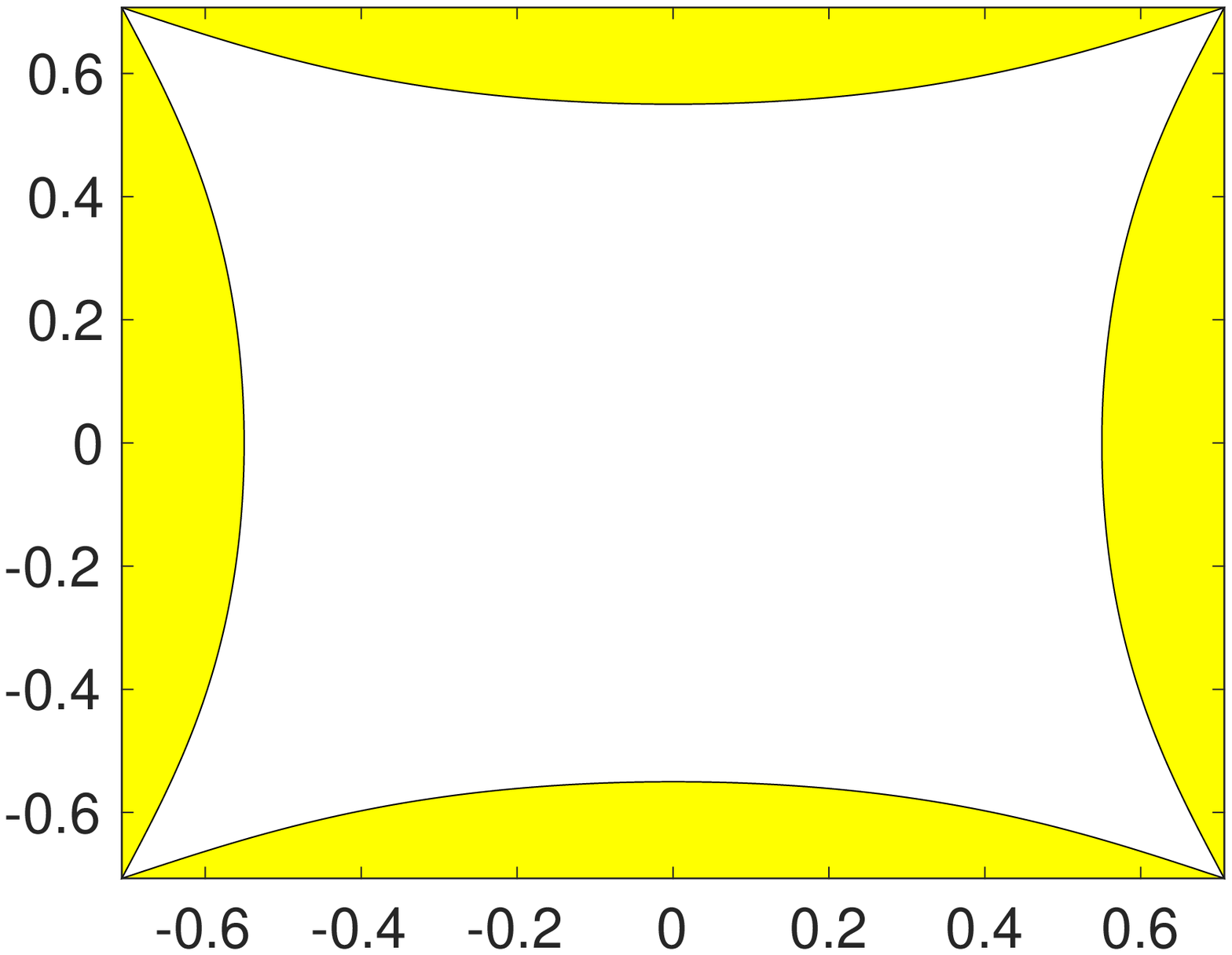}
\caption{
 Morphological hysteresis showing two different morphologies of phase
    arrangements in a square.  The hypothetical  coherent precipitate  
 with simple topology   and  square symmetry corresponds to 30\% volume fraction of phase 2 ($\Go=0.3$) is  shown in
    yellow. In the left panel nucleation starts at the center. In the right
    panel nucleation starts at the boundary. 
  }
\label{fig:square}
\end{figure}

The square symmetry of a problem also suggests
that we should look for a structure with square symmetry. Thus, if we assume that $A$  has a simple topology without singular points,  we  just need
to find the curve $\GG$ that joins the top and bottom ends of the right side of the square. Knowing $H_{+}(z)$ in all four regions adjacent to the
sides of the square gives the boundary values for $H_{-}(z)$ on $\Md A$.
The square symmetry of $A$ together with (\ref{complex}) implies that
\begin{equation}
H_{-}(iz)=-iH_{-}(z),\quad z\in\Md A.
\label{sqsym}
\end{equation}
 There is a holomorphic function $H_{-}$
in $A$ with given boundary values \IFF for every $n\ge 0$
\begin{equation}
\int_{\Md A}H_{-}(z)z^{n}dz=0.
\label{theeq}
\end{equation}
The integral in (\ref{theeq}) can easily be written as an integral along $\GG$
because of (\ref{sqsym}). If $n$ is not a multiple of 4 then the integral in
(\ref{theeq}) will evaluate to zero by virtue of (\ref{sqsym}) alone. If
$n=4k$, then we require that
\begin{equation}
\int_{\GG}H_{-}(z)z^{4k}dz=0,\quad k\ge 0.
\label{Gamma}
\end{equation}
On $\GG$ we have
$H_{-}(z)=\Go(\sqrt{2}-z)-\bra{z}$. Therefore, (\ref{Gamma}) is equivalent to
\begin{equation}
\int_{\GG}\bra{z}z^{4k}dz=\frac{(-1)^{k}i\Go}{(4k+1)(2k+1)},\quad k\ge 0.
\label{shapeq}
\end{equation}
If we  integrate by parts in (\ref{shapeq}):
\[
\int_{\GG}\bra{z}z^{4k}dz=\int_{\GG}\bra{z}\,d\!\left(\frac{z^{4k+1}}{4k+1}\right)=
-\int_{\GG}\frac{z^{4k+1}}{4k+1}d\bra{z}.
\]
and parametrize the curve $\GG$ by $z(t)=a(t)+it$,
$t\in[-1/\sqrt{2},1/\sqrt{2}]$. Then
$\bra{z'(t)}=z'(t)-2i$, 
we obtain
\[
\int_{\GG}\bra{z}z^{4k}dz=\frac{(-1)^{k+1}i}{(4k+1)(2k+1)}+\frac{2i}{4k+1}
\int_{-1/\sqrt{2}}^{1/\sqrt{2}}z(t)^{4k+1}dt.
\]
Thus, the equation (\ref{shapeq}) becomes
\[
\int_{-1/\sqrt{2}}^{1/\sqrt{2}}z(t)^{4k+1}dt=\frac{(-1)^{k}(\Go+1)}{2(2k+1)}.
\]
Due to the symmetry we have  $z(-t)=\bra{z(t)}$. Therefore, we finally
obtain
\begin{equation}
\re\left\{\int_{0}^{1/\sqrt{2}}(a(t)+it)^{4k+1}dt\right\}=
\frac{(-1)^{k}(\Go+1)}{4(2k+1)},\quad k\ge 0.
\label{fineq}
\end{equation}
In Fig.~\ref{fig:square} we show a numerical approximation of a hypothetically existing exact  solution of \eqref{fineq} with square symmetry and simple topology which should be taken at this point just as an indication of a general structure of the actual solution. It is presented to make a link with Fig.~\ref{fig:mhyst} where we show that  at a fixed value of the volume fraction $\Go$
two optimal configurations with the same energy always exist. As we have already seen in Fig.~\ref{fig:mhyst} , with the increase of $\Go$ the nucleus in the left panel in  Fig.~\ref{fig:square} (colored in yellow) will grow outward while the nucleus in the right panel (also colored in yellow)  will grow inward. At  $\Go=1$  both transformation  again  terminate as the second phase   completely takes over the whole domain.  

Observe now that $|a(t)+it|<1$ for all $t\in [0,1/\sqrt{2})$ and
$|a(t)+it|=1$ for $t=1/\sqrt{2}$. Thus, for large $k$ the principal
contribution to the integral (\ref{fineq}) comes only from the \nbh\ of
$t=1/\sqrt{2}$. Therefore, we approximate
\begin{equation}
  \label{corner}
a(t)=\nth{\sqrt{2}}+m(t-\nth{\sqrt{2}})+O((t-\nth{\sqrt{2}})^{2}),
\text{ as }t\to\nth{\sqrt{2}}.
\end{equation}
Substituting the leading term in (\ref{fineq}) and computing the integral explicitly
we obtain
\begin{equation}
  \label{asymp}
\frac{(-1)^{k}}{(4k+2)(m^{2}+1)} -
\frac{(1-m)^{4k+2}m}{2^{2k+1}(4k+2)(m^{2}+1)}
\approx\frac{(-1)^{k}(\Go+1)}{2(4k+2)}.
\end{equation}
If we choose
\begin{equation}
m=\sqrt{\frac{1-\Go}{1+\Go}}
\label{slopem}
\end{equation}
then for large $k$ we have equality in (\ref{fineq}) up to an exponentially
small error. Thus, we have determined the slope $m$ with which the curve $\GG$
enters the corner of the square.
We can integrate by parts in (\ref{fineq}), which we rewrite as
\[
\nth{4k+2}\re\left\{\int_{0}^{1/\sqrt{2}}\frac{\left[(a(t)+it)^{4k+2}\right]'}{i+a'(t)}dt\right\}=
\frac{(-1)^{k}(\Go+1)}{4(2k+1)}.
\]
We obtain
\begin{equation}
  \label{fineq1}
    \nth{4k+2}\re\left\{\frac{e^{i(4k+2)\pi/4}}{i+m}\right\}
-\nth{4k+2}\re\left\{\int_{0}^{1/\sqrt{2}}(a(t)+it)^{4k+2}\left[\nth{i+a'(t)}\right]'dt\right\}
  =\frac{(-1)^{k}(\Go+1)}{4(2k+1)}.
\end{equation}
But then we conclude that (\ref{fineq1}) is equivalent to
\begin{equation}
  \label{fineq0}
\re\left\{\int_{0}^{1/\sqrt{2}}(a(t)+it)^{4k+2}\left[\nth{i+a'(t)}\right]'dt\right\}=0,\qquad k\ge 0.
\end{equation}
We can now keep integrating by parts in (\ref{fineq0})
\[
\re\left\{\int_{0}^{1/\sqrt{2}}[(a(t)+it)^{4k+3}]'\nth{i+a'(t)}\left[\nth{i+a'(t)}\right]'dt\right\}=0.
\]
This implies that $a''(1/\sqrt{2})=0$. By induction we 
obtain $a^{(n)}(1/\sqrt{2})=0$ for all $n\ge 2$, showing the
breaking of analyticity of $a(t)$ near $t=1/\sqrt{2}$. This unexpected failure
of analyticity is our first hint that the solution we seek might not exist.

To make a more persuasive argument that this is the case
multiply the $k$th equation in (\ref{fineq}) by
$x^{4k+2}/(4k+1)!$ and sum, obtain the ``generating function equation''
\[
  \re\left\{\int_{0}^{1/\sqrt{2}}[\sinh(x(it+a(t)))+\sin(x(it+a(t)))]dt\right\}=
\frac{\Go+1}{2ix}(\cosh(\sqrt{i}x)-\cos(\sqrt{i}x)).
\]
Taking the real part we obtain
\begin{equation}
  \label{genf}
\int_{0}^{1/\sqrt{2}}[\cos(xt)\sinh(xa(t))+\sin(xa(t))\cosh(xt)]dt=
\frac{\Go+1}{x}\sinh\frac{x}{\sqrt{2}}\sin\frac{x}{\sqrt{2}}.
\end{equation}
The two sides of (\ref{genf}) represent entire function of $x$, and hence, the
equality is valid for all $x\in\bb{C}$.
It is convenient to rescale the equation: $u(s)=\sqrt{2}a(s/\sqrt{2})$, $z=x/\sqrt{2}$. Then
\begin{equation}
  \label{genf1}
\int_{0}^{1}[\cos(zs)\sinh(zu(s))+\sin(zu(s))\cosh(zs)]ds=\frac{\Go+1}{z}\sinh z\sin z.
\end{equation}
The function $u(s)$ is continuous and monotone increasing from $u(0)>0$ to 1
on $[0,1]$. It also satisfies $u'(0)=0$. We know that $u(s)\to s$,
when $\Go\to 0$ and we also know that $u(s)\sim 1+m(s-1)$, when
$s\approx 1$. Observe that $m\to 1$, when $\Go\to 0$, and $1+m(s-1)\to
s$. Let us then write $u(s)=u_{0}(s)+w(s)$, where $u_{0}(s)=1+m(s-1)$. Thus,
\[
\sinh(zu(s))=\sinh(zu_{0}(s))+2\sinh(zu_{0}(s))\sinh^{2}(zw(s)/2)+\cosh(zu_{0}(s))\sinh(zw(s)),
\]
and similarly,
\[
\sin(zu(s))=\sin(zu_{0}(s))-2\sin(zu_{0}(s))\sin^{2}(zw(s)/2)+\cos(zu_{0}(s))\sin(zw(s)),
\]
Observing that
\[
\int_{0}^{1}[\cos(zs)\sinh(zu_{0}(s))+\sin(zu_{0}(s))\cosh(zs)]ds=\frac{\Go+1}{z}\sinh z\sin z+R(z),
\]
where
\[
R(z)=\sqrt{1-\Go^{2}}\frac{\cos((1-m)z)-\cosh((1-m)z)}{2z},
\]
we obtain
\begin{multline}
\left.\int_{0}^{1}\right[\cos(zs)\cosh(zu_{0}(s))\sinh(zw(s))+\cosh(zs)\cos(zu_{0}(s))\sin(zw(s))+\\
\left.2\cos(zs)\sinh(zu_{0}(s))\sinh^{2}\left(\frac{zw(s)}{2}\right)
-2\cosh(zs)\sin(zu_{0}(s))\sin^{2}\left(\frac{zw(s)}{2}\right)\right]ds=\\
\sqrt{1-\Go^{2}}\frac{\cos((1-m)z)-\cosh((1-m)z)}{2z}.
\label{upert}
\end{multline}
If we now examine equation (\ref{upert}) in the limit $z\to+\infty$ we can
see that this is impossible. 

Indeed, the oscillatory terms
generate  at most polynomial decay at infinity, while the exponential terms
behave like $e^{zu_{0}(s)}$ attenuated by $\sin(w(s)z)$ or $\sinh(w(s)z)$. The
main contribution to such integrals comes from the vicinity of the maximizer
of $u_{0}(s)$, over the support $[0,\Ga]$ of $w(s)$. The integral will then have the
exponential growth $e^{(1-m+m\Ga)) z}$ possibly modulated by polynomially decaying
factors. However, the \rhs\ of (\ref{upert}) has the
exponential growth $e^{z(1-m)}$. Observing that $1-m=\min_{s\in[0,1]}u_{0}(s)$
we conclude that the integral on the \lhs\ of (\ref{upert}) cannot possibly
have that growth at infinity, unless the support of $w(s)$ is zero.  To summarize, we have  shown  that the equality
(\ref{upert}) cannot be satisfied  for all $m\in(0,1)$. We   note, however that $w(s)=0$ does satisfy (\ref{upert}) when $m=0$ or $m=1$. 

Finally, we remark that even if we had not assumed square symmetry of the solution our conclusion would not have changed. The reason is that equation (\ref{fineq}) corresponds to (\ref{theeq}) for $n=4k$. When $n$ has a different remainder mod 4 square symmetry  ensures that  the corresponding equation   is trivially satisfied. Breaking the symmetry simply
adds additional infinite systems on the extra degrees of freedom, ultimately
forcing one to assume  square symmetry which, as we have seen,  leads to  non-existence.
\def\cprime{$'$} \ifx \cedla \undefined \let \cedla = \c \fi\ifx \cyr
  \undefined \let \cyr = \relax \fi\ifx \cprime \undefined \def \cprime
  {$\mathsurround=0pt '$}\fi\ifx \prime \undefined \def \prime {'}
  \fi\def\Ya{Ya}

\end{document}